\documentclass{amsart}
\usepackage{amssymb}
\usepackage{color}
\usepackage{graphicx, tikz}
\usetikzlibrary{math}

\usepackage{pgfplots}
\usepgfplotslibrary{fillbetween}
\pgfplotsset{compat=1.17} 

\usepackage{mathtools}
\mathtoolsset{showonlyrefs}

\usepackage[giveninits=true, maxnames=10]{biblatex}

\newcommand{\R}{\mathbb{R}}
\newcommand{\indic}{1\!\!1}

\DeclareMathOperator{\sign}{sign}
\DeclareMathOperator{\arctanh}{arctanh}
\DeclareMathOperator{\Ran}{Ran}

\newtheorem{theo}{Theorem} 

\newtheorem{lemma}{Lemma}[section]
\newtheorem{prop}[lemma]{Proposition}
\newtheorem{corol}[lemma]{Corollary}

\newtheorem{conj}{Conjecture}

\theoremstyle{remark}
\newtheorem{remark}[lemma]{Remark}

\bibliography{toto.bib}
\renewbibmacro{in:}{}
\AtEveryBibitem{
 \clearlist{address}
 \clearfield{date}
\clearfield{eprintclass}
 \clearfield{isbn}
 \clearfield{issn}
 \clearlist{location}
 \clearfield{month}
 \clearfield{series}
 \clearfield{url}
 \clearfield{doi}
 \clearfield{Language}

 \ifentrytype{book}{}{
  \clearlist{publisher}
  \clearname{editor}
 
 }
}

\begin{document}
\author[T.~Duyckaerts]{Thomas Duyckaerts}
\address{LAGA (UMR 7539), Universit\'e Sorbonne Paris Nord, Institut Galil\'ee, 99 avenue Jean-Baptiste Cl\'ement, 93430 Villetaneuse, France.}
\email{duyckaer@math.univ-paris13.fr}
\author[G.~Negro]{Giuseppe Negro}
\address{  
Instituto Superior Técnico\\
Avenida Rovisco Pais\\ 
1049-001 Lisboa, Portugal.} 
\email{giuseppe.negro@tecnico.ulisboa.pt}

 \title[Asymptotically self-similar global solutions]{Global solutions with asymptotic self-similar behaviour for the cubic wave equation}

 \begin{abstract}
    We construct a two-parameter family of explicit solutions to the cubic wave equation on $\mathbb{R}^{1+3}$. Depending on the value of the parameters, these solutions either scatter to linear, blow-up in finite time, or exhibit a new type of threshold behaviour which we characterize precisely.
 \end{abstract}
 \maketitle
 
 \section{Introduction}
 \subsection{Background}
 Consider the wave equation in space dimension $3$ with a power-like, focusing nonlinearity
 \begin{equation}
 \label{gNLW}
 \partial_t^2u-\Delta u=|u|^{p-1}u, \quad x\in \R^3, \quad t\in I,
 \end{equation} 
 where $I$ is an interval such that $0\in I$, with initial data
 \begin{equation}
     \label{gNLWID}
     \boldsymbol{u}_{\restriction t=0}=(u_0,u_1)\in \mathcal{H}^s,
 \end{equation}
where
 \begin{equation}\label{eq:notation}
        \mathcal{H}^{s}=\dot{H}^{s}(\R^3)\times \dot{H}^{s-1}(\R^3),\quad  \boldsymbol{u}(t)=(u(t, \cdot), \partial_t u(t, \cdot)),
\end{equation}
and $\dot{H}^s$ denotes the usual homogeneous $L^2$ Sobolev space. The equation~\eqref{gNLW} is invariant by scaling: if $\lambda>0$ and $u$ is a solution on a time interval $I$, so is $u_{\lambda}$, defined by
$$u_{\lambda}(t,x)=\lambda^{\frac{2}{p-1}}u(\lambda t,\lambda x),$$
on the interval $\frac{1}{\lambda}I$. For $p\ge 3$, the initial value problem \eqref{gNLW},  \eqref{gNLWID} is locally well-posed in $\mathcal{H}^s$ for any $s\geq s_c$, where $s_c=\frac 32-\frac{2}{p-1}$ is the critical Sobolev exponent, i.e. the exponent such that $\left\|\boldsymbol{u}_{\lambda}(0)\right\|_{\mathcal{H}^{s_c}}=\|\boldsymbol{u}(0)\|_{\mathcal{H}^{s_c}}$ (with some exceptions when $p$ is not an odd integer and $s$ is large; see for example~\cite[§3.3]{Tao06BO}). If the initial data is in $\mathcal{H}^1\cap \mathcal{H}^{s_c}$, then the \textit{energy}
\begin{equation}
    \label{defE}
    E(\boldsymbol{u}):=\frac{1}{2}\int_{\mathbb R^3}|\nabla u(x)|^2\, dx+\frac{1}{2}\int_{\mathbb R^3} (\partial_tu(x))^2\, dx -\frac{1}{p+1}\int_{\mathbb R^3}|u(x)|^{p+1}\, dx
\end{equation}
 is well-defined and conserved with time. 
 
 We are interested in the global dynamics of solutions of \eqref{gNLW}, and especially solutions with initial data in the critical space $\mathcal{H}^{s_c}$. A first example is given by \textit{scattering} solutions, that is solutions $u$ that are global for positive times and such that there exists a solution $u_L$ of the free wave equation
 \begin{equation}
     \label{FW}
     \partial_t^2u_L-\Delta u_L=0
 \end{equation}
 satisfying $\lim_{t\to\infty}\left\|\boldsymbol{u}(t)-\boldsymbol{u}_L(t)\right\|_{\mathcal{H}^{s_c}}=0$. It is well-known that if $p\geq 3$, the set of initial data $(u_0,u_1)$ in $\mathcal{H}^{s_c}$ that lead to scattering solutions is open in $\mathcal{H}^{s_c}$; in other words, scattering is a stable behaviour. Since this set of initial data obviously contains the null data $\boldsymbol{0}$, solutions with small data in $\mathcal{H}^{s_c}$ scatter.
 
 Other known solutions of \eqref{gNLW} are \textit{self-similar} solutions, i.e. solutions of the form 
 $u(t,x)=(T\pm t)^{\frac{2}{p-1}} \varphi\left((T\pm t) x\right)$ for some profile $\varphi$ which solves a certain explicit partial differential equation. The first example of a self-similar solution,  corresponding to a constant profile $\varphi$, is $u(t, x)=y(t)$ where $y$ solves the ordinary differential equation $y''=y^p$; so, up to a time translation, $y(t)=c_p t^{-\frac{2}{p-1}}$, with $c_p=\left(\frac{2(p+1)}{(p-1)^2}\right)^\frac{1}{p-1}$, $t>0$. Other self-similar profiles have been constructed in \cite{BizonBreitenlohnerMaisonWasserman} for $p=3$ and in \cite{BiMaWa07} for $p\geq 5$. These profiles do not yield solutions with finite energy, or with initial data in the critical space $\mathcal{H}^{s_c}$, although the existence of such solutions is not theoretically excluded in general  (see \cite{KavianWeissler90} for the proof that there is no self-similar, finite energy solutions in the energy-supercritical case, and discussions on the subject). Nevertheless, the self-similar profiles and especially the ODE solution are expected to play an important role in the dynamics of singular solutions of \eqref{gNLW}.
 
 Indeed, it was shown by F. Merle and H. Zaag that any blow-up solution of~\eqref{gNLW} in space dimension $1$ with $p>1$ converges in the past wave cone arising from the blow-up point to the constant $\left(\frac{2(p+1)}{(p-1)^2}\right)^\frac{1}{p-1}$, up to a self-similar rescaling and possibly a Lorentz transformation. In other words, any blow-up solution is close, up to symmetries, to the ODE solution in this wave cone (see \cite[Theorem 2 and Corollary 4]{MeZa07}). The same property is expected to hold in other dimension and in particular in space dimension $3$. This has been proven for~\eqref{gNLW} in the subconformal case ($p<3$ in dimension $3$), with additional assumptions (see \cite{MerleZaag16}), and observed numerically for  $p\in \{3,5,7\}$ in \cite{BizonChmajTabor04}. This ODE blow-up is stable by small perturbations (see again \cite{BizonChmajTabor04} for numerics and \cite{DonningerSchorkhuber12,MeZa15,DonningerSchorkhuber16,Donninger17}). 
 
 In this work, we are mainly interested in global solutions with initial data in $\mathcal{H}^{s_c}$ that do not scatter to a linear solution. Numerical and theoretical works (see \cite{BizonChmajTabor04}, \cite{BizZen09}, \cite{KriegerNahas15}) suggest that these solutions are unstable, at a threshold between the scattering and ODE blow-up behaviours described above.

 These solutions are quite well-understood for the energy-critical power $p=5$, corresponding to $s_c=1$. In this case, there exist \textit{stationary} solutions of \eqref{gNLW}, that is solutions $u(t, x)=Q(x)$ where 
 \begin{equation}
 \label{Ell}
 -\Delta Q=|Q|^{p-1}Q, \quad Q\in \dot{H}^{s_c}(\mathbb R^3).
 \end{equation}
 The only radial solution of \eqref{Ell} for $p=5$, up to scaling and sign change, is the \textit{ground state} $W=\left(1+\frac{|x|^2}{3}\right)^{-1/2}$. Taking  Lorentz transforms of the stationary solutions one obtains \textit{solitary waves}, travelling at a fixed velocity.
 
 These stationary solutions and solitary waves play a crucial role in the dynamics of global solutions. Indeed it was proved in \cite{DuKeMe13} that, in the radial case, all solutions which are global in the future are bounded in the \emph{energy space} $\mathcal{H}^1$, and can be written asymptotically, as $t\to\infty$, as a finite sums of decoupled rescaled ground states, plus a radiation term (solution of the free wave equation \eqref{FW}) and a term which goes to $0$ in $\mathcal{H}^1$. (Weaker version of this result exist without symmetry assumption on the solutions \cite{DuJiKeMe17a}). This type of behaviour is called ``soliton resolution'', by analogy with the soliton resolution known for completely integrable partial differential equations such as Korteweg-de Vries (see \cite{EcSc83}). Let us mention that in this case $p=5$, there also exist solutions blowing up in finite time and remaining bounded in the energy space. These so-called \emph{type II} blow-up solutions have an analogous asymptotic behaviour as the global non-scattering solutions. 
 
 The situation is quite different when $p \neq 5$. Indeed, it holds that~\eqref{Ell} does not have any nonzero solution $Q\in L^{p+1}\cap \dot{H}^1$ in this case (this is a consequence of the Pohozaev identity, see e.g. \cite[Proposition~1]{BeLi83a}). Furthermore, it is proved in the radial case (see \cite{DuKeMe14} for $p>5$, \cite{Shen14} for $3<p<5$ and \cite{DodsonLawrie15} for $p=3$) that for $p\geq 3$, 
 $p\neq 5$, any solution to~\eqref{gNLW} such that 
\begin{equation}
\label{limsup}
 \limsup_{t\to\infty}\|\boldsymbol{u}(t)\|_{\mathcal{H}^{s_c}}<\infty
\end{equation} 
 is a scattering solution. When $p>5$, one can weaken \eqref{limsup}, see \cite{DuyckaertsRoy17,DuyckaertsYang18}. These two facts exclude the possibility of a soliton resolution when $p\neq 5$. Actually, very few is known in this case, and even the existence of global, non-scattering solutions with initial data in $\dot{H}^{s_c}$ is an open question. 
 
 In this paper, we will construct such a solution in the physically relevant case $p=3$, that is the equation
  \begin{equation}
 \label{NLW}
 \partial_t^2u-\Delta u=u^3, \quad x\in \R^3,
 \end{equation}
 with initial data 
 \begin{equation}
     \label{ID}
     \boldsymbol{u}_{\restriction t=0}=(u_0,u_1)\in \mathcal{H}^{1/2}\cap \mathcal{H}^1.
 \end{equation}
 The critical exponent is $s_c=1/2$, and the equation is invariant by conformal transformations.  In particular, if $u$ is a solution of \eqref{NLW}, so is
 \begin{equation}\label{eq:conformal_inversion}
     v(t, x)=\frac{1}{t^2-|x|^2} u\left(\frac{t}{t^2-|x|^2},\frac{x}{t^2-|x|^2}\right),
 \end{equation}
 (at least formally).  Previous works (see \cite{BizZen09,DonningerZenginoglu}) suggest that the solution $\sqrt{2}/t$ plays a role in the asymptotic dynamics of threshold solutions; note that this solution is invariant under~\eqref{eq:conformal_inversion}.
 
Here we will construct two one-parameter families of smooth radial solutions of \eqref{NLW}, \eqref{ID}, which are global and non-scattering for positive times. As $t\to\infty$:
 \begin{itemize}
     \item the critical Sobolev norm $\mathcal{H}^{1/2}$ of these solutions blows up as a power of $\log t$ as $t\to\infty$. The norms $\mathcal{H}^s$, $s>1/2$ remain bounded.
\item the solutions are asymptotically close to one of the self-similar solutions $\pm \frac{\sqrt{2}}{t}$ in the interior of the wave cone $\{|x|<t\}$, 
\item they behave as a linear solution in the exterior of this wave cone.
 \end{itemize}
 Among these solutions, exactly $2$ of them are odd in time, and $2$ of them are even in time, giving all the possible combinations of behaviours asymptotic to $\pm \sqrt{2}/t$ inside the wave cone both in the future and in the past. The other solutions either scatter or blow-up in finite time in the past.
 
 These solutions are to our knowledge the first theoretical examples of global solutions of \eqref{NLW}, \eqref{ID} (and also of the general non-linear wave equation \eqref{gNLW}) with initial data in the critical Sobolev space which are not bounded in this Sobolev space. This behaviour is sometimes referred to as ``blow-up at infinity" or ``grow-up at infinity". In accordance to the numerical work \cite{BizZen09}, we believe that these solutions have a generic asymptotic behaviour in the class of threshold solutions (that is, solutions that are at the boundary of ODE blow-up and scattering). The main goal of this article is to give a complete description of this asymptotic behaviour, which we see as an important step to understand the dynamics of nonscattering global solutions and of threshold solutions of \eqref{NLW}.
 
 \subsection{A two-parameter family of solutions}
 
We will construct a two-parameter family of solutions to~\eqref{NLW}. The aforementioned global non-scattering solutions will be at the boundary between blow-up and scattering solutions. More precisely,  we consider the family $u_{X, Y}=u_{X, Y}(t, x)$ of the solutions to the cubic wave equation~\eqref{NLW} with initial data
 \begin{equation}\label{eq:uAB_initial}
    \begin{array}{cc}
        \displaystyle
        \boldsymbol{u}_{X, Y}(0,x)=\left(\frac{2X}{1+\lvert x \rvert^2}, 
        \frac{4Y}{(1+\lvert x \rvert^2)^2}\right),
        & \text{where }(X, Y)\in \mathbb R^2.
    \end{array}
\end{equation}
As we will see, for each $(X, Y)\in\mathbb R^2$, $u_{X, Y}$ is smooth and it is defined on an open set in $\mathbb R^{1+3}$ that contains the initial time slice $\{0\}\times \mathbb R^3$. In particular, there is a maximal time interval of existence which we denote by 
\begin{equation}\label{eq:maximal_time_interval}
    (t_-(X, Y), t_+(X, Y))\subset[-\infty, \infty].
\end{equation}
For all $t$ in this interval and all $s>-1/2$, it holds that $\boldsymbol{u}_{X, Y}(t)\in \mathcal{H}^s$. Our main theorem describes all possible long-time behaviours of these solutions, in terms of a threshold function $\beta=\beta(X)$, obtained by solving an appropriate integral equation which we will describe precisely in Section~\ref{sec:ODE}. 
\begin{theo}\label{thm:main}
 There is a strictly decreasing smooth function $\beta\colon\mathbb R\to \mathbb R$ satisfying 
 \begin{equation}\label{eq:beta_props}
    \begin{array}{cc}
        \displaystyle \lim_{X\to \pm \infty} \beta(X)=\mp \infty, & \beta(X)>-\beta(-X), 
    \end{array}
\end{equation}
and such that for each $(X, Y)\in\mathbb R^2$, one and only one of the following occurs;
 \begin{itemize}
  \item[(Blow-up)] If $Y>\beta (X)$ or $Y<-\beta(-X)$, then $t_+(X, Y)<\infty$ and 
  \begin{equation}
    \begin{array}{ccc}
      \lVert \boldsymbol{u}_{X, Y}(t)\rVert_{\mathcal H^s}~\to~\infty& \text{as } t\to t_+(X, Y),& \text{ for }s\ge \frac12.
    \end{array}
  \end{equation}    
 
  \item[(Scattering)] If $-\beta(-X)<Y<\beta(X)$, then  $t_+(X, Y)=+\infty$ and there is a smooth solution $v_{X,Y}^+$ to the linear wave equation $\partial^2_t v_{X,Y}^+ = \Delta v_{X, Y}^+$ such that 
  \begin{equation}\label{eq:scattering}
    \begin{array}{cc}\displaystyle
        \lVert \boldsymbol{u}_{X, Y}(t)- \boldsymbol{v}_{X, Y}^+(t)\rVert_{\mathcal H^{1/2}}\to 0, & \text{as }t\to\infty.
    \end{array}
\end{equation}
  \item[(Threshold)] If $Y=\beta(X)$ or $Y=-\beta(-X)$, then  $t_+(X, Y)=+\infty$ and $\lVert \boldsymbol{u}_{X, Y}(t)\rVert_{\mathcal H^{1/2}}\,\to\,\infty$ as $t\to \infty$. 
 \end{itemize}
\end{theo}
\begin{figure}
    \centering
    \includegraphics{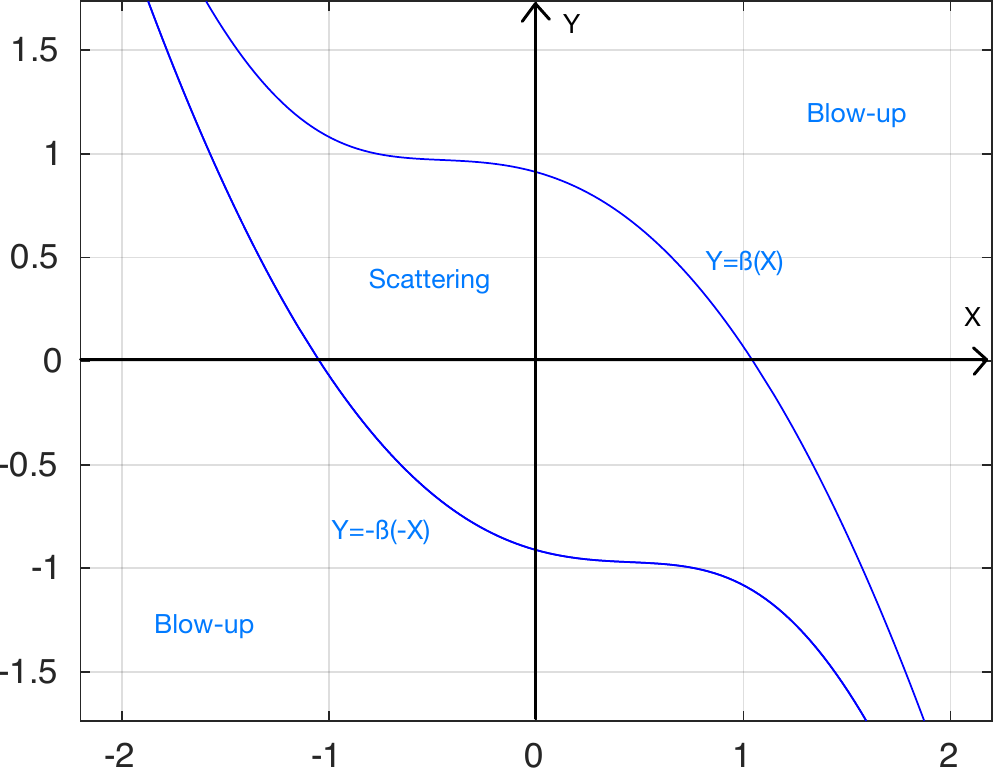}
    \caption{Illustration of Theorem~\ref{thm:main}; blow-up, scattering and asymptotically self-similar behaviour for $t>0$.}
    \label{fig:forward_time_phase_space}
\end{figure}
Figure~\ref{fig:forward_time_phase_space} contains a plot of this threshold function $\beta$. More precisely, $u_{X, Y}$ is radially symmetric and is given by 
\begin{equation}\label{eq:explicit_solution_formula}
u_{X,Y}(t,x)=\Omega(t,\lvert x \rvert) U_{X, Y}(\arctan(t+\lvert x \rvert)+\arctan(t-\lvert x \rvert)),
\end{equation}
 where $\Omega(t,r)=\frac{2}{\sqrt{1+(t+r)^2}\sqrt{1+(t-r)^2}}$, and $U(s)=U_{X, Y}(s)$ solves
 $$\ddot{U}+U=U^3,\quad U(0)=X,\; \dot U(0)=Y.$$
In particular, it follows immediately from these formulas that, if $(X, Y)\ne (X', Y')$, then the corresponding $u_{X, Y}$ and $u_{X', Y'}$ are not related by any of the symmetries of~\eqref{NLW}. We recall that the group of such symmetries is generated by scaling, space-time translations and Lorentz transformations.

\begin{remark}\label{rem:completely_explicit}Some of the solutions $u_{X, Y}$ are completely explicit. For $\lvert A\rvert<\sqrt 2$ and $\theta\in\mathbb R$,  the function
\begin{equation}\label{eq:explicit_global_intro}
    \begin{array}{ccc}
        U(s)=2 A \operatorname*{sn}(\omega s + \theta, k^2), & \text{where }\omega^2=1-\frac{A^2}{2},\,  k^2=\frac{A^2}{2-A^2},
    \end{array}
\end{equation}
and where $\operatorname*{sn}$ denotes Jacobi's elliptic sine, is a global solution to $\ddot{U}+U=U^3$. The formula~\eqref{eq:explicit_solution_formula} thus gives a corresponding family of completely explicit solutions to~\eqref{NLW}, which are part of the (Scattering) case of Theorem~\ref{thm:main}. On the other hand, for $T_+\in(0, \pi)$, the function 
\begin{equation}\label{eq:explicit_blowup_ODE}
    U(s)=\frac{\sqrt 2}{\sin(T_+-s)}
\end{equation}
is a solution to $\ddot{U}+U=U^3$ that blows up at $s=T_+$. The corresponding (Blow-up) solution $u$ to~\eqref{NLW} (given by~\eqref{eq:explicit_solution_formula}) has the following simple expression:
\begin{equation}\label{eq:u_a_b}   
    \begin{array}{cc}\displaystyle
        u(t, r)=\frac{2\sqrt 2}{a(1+r^2-t^2)-2bt},& \text{ where }a=\sin T_+, b=\cos T_+;
    \end{array}
\end{equation}
see the proof of the forthcoming Proposition~\ref{prop:blow_up_attractor}.
\end{remark}
We will give a precise description of the (Threshold) solutions below. One important property is that $u_{X, Y}$ is asymptotically close to the self-similar solution $(t, x)\mapsto~\sqrt 2 / t$, in the interior of the forward light cone $t>\lvert x \rvert$. Because of this, we refer to the (Threshold) solutions as \emph{asymptotically self-similar}.

\begin{remark}\label{rem:neg_times}     
     It is clear from~\eqref{NLW} that 
     \begin{equation}\label{eq:X_minus_X}
         u_{X, Y}(t,x)=u_{X,-Y}(-t,x).
    \end{equation}
     Combining this observation with Theorem~\ref{thm:main}, we see that for negative times one and only one of the following occurs:
     \begin{itemize}
  \item[(Blow-up)] If $Y>\beta (-X)$ or $Y<-\beta(X)$, then $\lvert t_-(X, Y)\rvert <\infty$ and 
  \begin{equation}
    \begin{array}{ccc}
      \lVert \boldsymbol{u}_{X, Y}(t)\rVert_{\mathcal H^s}~\to~\infty& \text{as } t\to t_-(X, Y),& \text{ for }s\ge \frac12.
    \end{array}
  \end{equation} 
  \item[(Scattering)] If $-\beta(X)<Y<\beta(-X)$, then  $t_-(X, Y)=-\infty$ and there is a smooth solution $v_{X,Y}^-$ to the linear wave equation $\partial^2_t v_{X,Y}^- = \Delta v_{X, Y}^-$ such that 
  \begin{equation}\label{eq:neg_scattering}
    \begin{array}{cc}\displaystyle
        \lVert \boldsymbol{u}_{X, Y}(t)- \boldsymbol{v}_{X, Y}^-(t)\rVert_{\mathcal H^{1/2}}\to 0, & \text{as }t\to-\infty.
    \end{array}
\end{equation}
  \item[(Threshold)] If $Y=\beta(-X)$ or $Y=-\beta(X)$, then  $t_-(X, Y)=-\infty$ and $\lVert \boldsymbol{u}_{X, Y}(t)\rVert_{\mathcal H^{1/2}}\,\to\,\infty$ as $t\to -\infty$.
 \end{itemize}
\end{remark}
\begin{figure}
    \centering
    \includegraphics{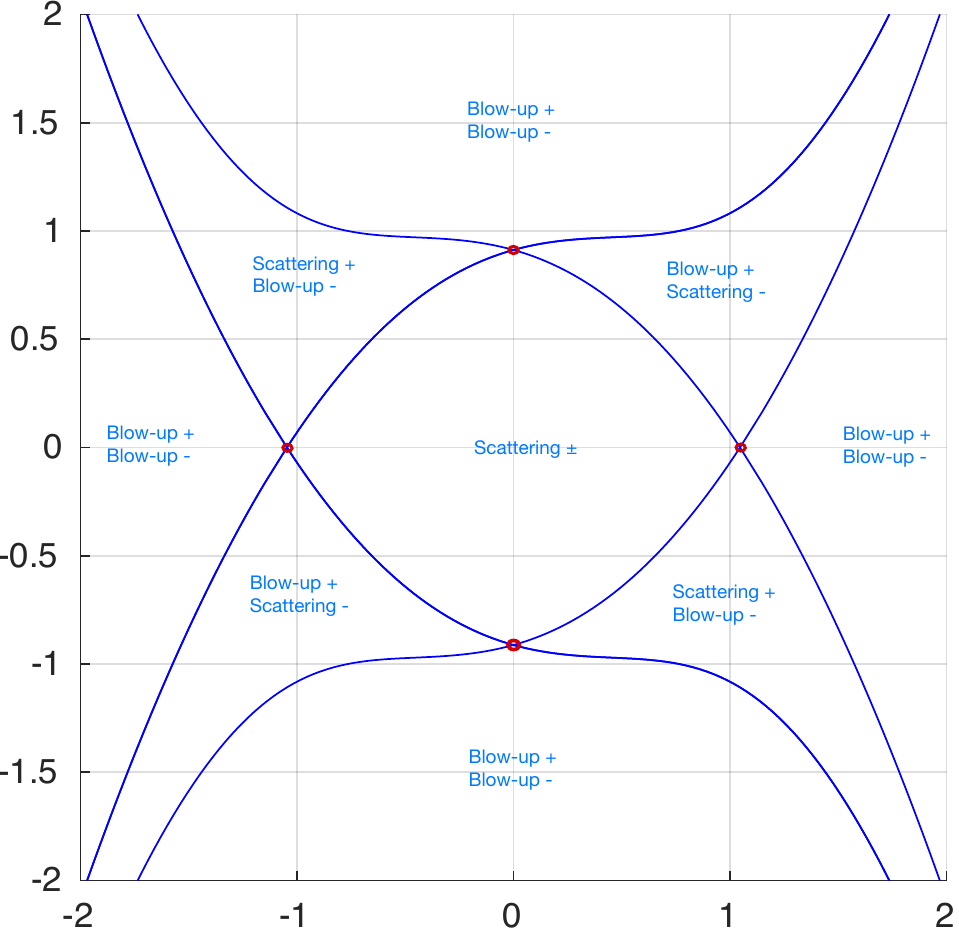}
    \caption{The nine possible behaviors in the plane $(X, Y)$. The dots mark the four unique points that give self-similar behavior at $\pm\infty$.}
    \label{fig:CompletePhaseSpace}
\end{figure}


The properties~\eqref{eq:beta_props} of the threshold function $\beta$ imply that the three sets defined for positive times and negative times have nonempty intersections, as depicted in Figure~\ref{fig:CompletePhaseSpace}. Thus all of the $9$ combined behaviours are possible. There are four remarkable solutions, whose initial data correspond to the dots in the figure. We have already mentioned these solutions, which are either even or odd in $t$ and are asymptotically self-similar both as $t\to\infty$ and as $t\to -\infty$; thus these solutions are \emph{homoclinic}. 

For the other values of $X$, $u_{X,\beta(X)}$ and $u_{X, -\beta(-X)}$ are \emph{heteroclinic}; they connect the asymptotic self-similar behaviour at $t\to\infty$ with another one (scattering or self-similar finite time blow-up) for negative $t$. Examples of heteroclinic orbits for \eqref{gNLW} were constructed in \cite{DuMe08} in the energy-critical case $p=5$. These solutions link the ground state of the equation with a blow-up or scattering behaviour. For the same equation, a nine-set classification, similar to Figure \ref{fig:CompletePhaseSpace}, but on the whole energy space $\mathcal{H}^1$, was obtained in \cite{KrNaSc15}. We remark that in the energy-critical case of these papers~\cite{DuMe08, KrNaSc15}, the threshold behaviour is given by solutions that are asymptotically close to the ground state $W$, possibly rescaled.

\subsection{Asymptotically self-similar solutions}
We next describe  the asymptotically self-similar threshold solutions (case (Threshold) in Theorem~\ref{thm:main}).
\begin{theo}[Self-similar behaviour inside the wave cone]
 \label{thm:self0}
Let $(X,Y)\in \mathbb{R}^2$ with $Y=\beta(X)$ or $Y=-\beta(-X)$. Denote $u=u_{X,Y}$. Then for all $r\ge 0$,
\begin{equation}
    \label{pointwiseCV}
    \limsup_{t\to\infty} \left( t^3\left| u(t,r)-\frac{\sqrt{2}}{t}\right|+t^4\left| \partial_t u(t,r)+\frac{\sqrt{2}}{t^2}\right|\right)<\infty.
\end{equation}
Furthermore, for all $p>\frac 32$, there exists $C>0$ such that 
\begin{equation}
    \label{LpCV0}
\left\|u(t)-\frac{\sqrt{2}}{t}\indic_{\{|x|\leq |t|\}}\right\|_{L^p(\mathbb R^3)}\leq C t^{\frac{2}{p}-1}.
\end{equation}
and, for all $p\geq 1$ and $\alpha\in (0,1)$, 
\begin{gather}
    \label{LpCV1}
    \left\|u-\frac{\sqrt{2}}{t}\right\|_{L^p(\{|x|\leq t-t^{\alpha}\})}\leq Ct^{\frac{2}{p}-1}t^{-\alpha\left(3-\frac{1}{p}\right)}\\
    \label{LpCV2}
    \left\|\partial_tu+\frac{\sqrt{2}}{t^2}\right\|_{L^p(\{|x|\leq t-t^{\alpha}\})}\leq Ct^{\frac{2}{p}-1}t^{-\alpha\left(2-\frac{1}{p}\right)}.
\end{gather}
\end{theo}
We note that $\left\lVert\frac{\sqrt{2}}{t}\right\|_{L^p(\{|x|\leq t\})}\approx t^{\frac{3}{p}-1}$,  which is larger than the right hand-side of \eqref{LpCV0}, \eqref{LpCV1}, and that $\left\lVert\frac{\sqrt{2}}{t^2}\right\|_{L^p(\{|x|\leq t\})}\approx t^{\frac{3}{p}-2}$, which is larger than the right-hand side of \eqref{LpCV2} for $\alpha>\frac{p-1}{3p-1}$. Thus our estimates in Theorem~\ref{thm:self0} are meaningful.

We next study the behaviour of threshold solutions in a close neighborhood of the wave cone.
\begin{theo}
\label{thm:self0'}
Let $(X,Y)\in \mathbb{R}^2$ with $Y=\beta(X)$ or $Y=-\beta(-X)$. Denote $u=u_{X,Y}$. Then there exists a solution $v_L$ of the linear wave equation \eqref{FW}, with radial initial data $(v_0,v_1)\in \cap_{s>1/2} \mathcal{H}^{s}$, $(v_0,v_1)\notin \mathcal{H}^{1/2}$ such that for all $A\in \mathbb{R}$
\begin{equation}
\label{asymptotic-linear}
\lim_{t\to\infty}\int_{|x|>t+A} \big|\nabla (u(t,x)-v_L(t,x))\big|^2+\big(\partial_t(u(t,x)-v_L(t,x))\big)^2\,dx=0.
\end{equation}
\end{theo}
Note that if $u$ is a scattering solution of \eqref{NLW} and \eqref{ID}, then \eqref{asymptotic-linear} holds for a linear solution $v_L$ with initial data in $\mathcal{H}^{1/2}\cap \mathcal{H}^{1}$. The fact that $(v_0,v_1)\notin \mathcal{H}^{1/2}$ is linked to the transition to the self-similar solution $\sqrt{2}/t$ in the interior of the wave cone $\{|x|<t\}$ close to the boundary $|x|=t$, see Remark \ref{rem:transition} below. The existence of a linear profile $v_L$ satisfying \eqref{asymptotic-linear} is standard for global solutions of \eqref{gNLW} which are bounded in the critical Sobolev space, and is the first step toward the asymptotic description of the solution: see  e.g. \cite{DuKeMe13}, \cite{DuKeMe14} and \cite{Shen14} for the radial case, and also \cite{DuKeMe19} for the energy-critical power without symmetry assumption.

Theorem \ref{thm:self0} is coherent with Conjecture 3 of the work of Bizo\'n and Zengino\v{g}lu \cite{BizZen09}, which predicted that solutions at the threshold between scattering and stable blow-up should satisfy an estimate similar to \eqref{pointwiseCV}. Note however that strictly speaking, the estimate predicted by \cite{BizZen09}, which amounts to 
$$\limsup_{t\to\infty} t^4\left|u(t,r)-\frac{\sqrt{2}}{t}\right|<\infty$$
is false in our case, since we can show that $u(t,r)-\sqrt{2}/t$ is exactly of order $1/t^3$. (We remark that we will also discuss below the connection between Conjecture 2 of \cite{BizZen09} and our (Blow-up) solutions).

The work~\cite{DonningerZenginoglu} is related to the previous conjecture. Indeed in~\cite{DonningerZenginoglu} Donninger and Zengino\v{g}lu constructed a manifold of solutions asymptotically close to the self-similar solution $\sqrt{2}/t$ inside the wave cone for $t\to \infty$. However, to tackle with the fact that $\sqrt{2}/t\notin \dot{H}^{1/2}(\mathbb{R}^3)$ at fixed $t$, they consider a different initial value problem, with initial data on a spacelike hyperboloid. These solutions do not a priori correspond to solutions of the Cauchy problem \eqref{NLW}, \eqref{ID}.
Our construction shows that this solution $\sqrt{2}/t$, despite its lack of decay at infinity, appears in the asymptotics of solutions of the usual Cauchy problem for \eqref{NLW}; in particular, note that we prescribe the initial data at $t=0$ and require spatial decay at infinity. 

In view of the works cited above, we believe that general global non-scattering have a similar asymptotic behaviour as the (Threshold)  solutions $u_{X,\beta(X)}$ and $u_{X, -\beta(-X)}$. We thus conjecture: 
\begin{conj}
\label{conj:global}
Let $u$ be a solution of \eqref{NLW}, \eqref{ID} defined for $t\in [0,\infty)$ and that does not scatter to a linear solution as $t\to\infty$. Then there exists a self-similar solution $S(t,x)=\frac{1}{t}\varphi\left(\frac{x}{t}\right)$ of \eqref{NLW}, such that
\begin{equation}
\label{self-similarS}
    \lim_{t\to\infty}\left\|u(t)-S(t)\right\|_{L^3(\{|x|<t-\sqrt{t}\}}+\left\|\partial_tu(t)+\partial_tS(t)\right\|_{L^{3/2}(\{|x|<t-\sqrt{t}\}) }=0.
\end{equation}
Furthermore, there exists a solution $v_L$ of the linear wave equation, with initial data $\boldsymbol{v}_L(0)\in \bigcap_{1/2<s\leq 1}\mathcal{H}^s$, $\boldsymbol{v}_L(0)\notin \mathcal{H}^{1/2}$, such that \eqref{asymptotic-linear} holds.
\end{conj}
Of course, the work \cite{BizZen09} suggest that we should have $S(t)=\pm\sqrt{2}/t$ for generic global non-scattering solution of \eqref{NLW}, but in full generality, we cannot exclude the appearance of other self-similar solutions, such as the ones constructed in \cite{BizonBreitenlohnerMaisonWasserman}. We also conjecture that the structure of the set of solutions asymptotically close to $\sqrt{2}/t$ inside the wave cone is similar to the one obtained in \cite{DonningerZenginoglu}:
\begin{conj}
\label{conj:manifold}
The set of initial data of solutions of \eqref{NLW}, \eqref{ID} such that \eqref{self-similarS} holds with $S(t)=\sqrt{2}/t$ is a submanifold of codimension $1$ of $\mathcal{H}^1\cap \mathcal{H}^{1/2}$, that separates scattering and ODE blow-up.
\end{conj}

The role of self-similar solutions in the threshold dynamics for equation \eqref{gNLW} when $p\neq 5$ (or more generally, in space dimension $N$, when $p\neq \frac{2N}{N-2}$) was highlighted in several previous works:
\begin{itemize}
    \item For the supercritical power $p=7$, the article \cite{BizonChmajTabor04} gives numerical evidence that, in this case, solutions at the threshold are blow-up solutions with an unstable self-similar profile (distinct from the stable ODE solution).
\item 
In higher space dimensions the papers \cite{GlMaSc20} and \cite{GlogicSchorkhuber21}  (for the cubic wave equation) and \cite{CsGlSc21P} (for the quartic wave equation) have exhibited an explicit unstable self-similar solution. In both cases, the power is energy super-critical. These works also show (theoretically and numerically) that at least in some cases, this solution has the generic threshold behaviour. Let us underline that the threshold solutions in \cite{BizonChmajTabor04}, \cite{GlMaSc20}, \cite{GlogicSchorkhuber21} and \cite{CsGlSc21P} are all finite-time blow-up solutions, in contrast with the solutions $u_{X,\beta(X)}$ and $u_{X, -\beta(-X)}$ that we construct here.
\item Other asymptotically self-similar solutions were constructed by Krieger and Schlag in \cite{KriegerSchlag17} for the equation \eqref{gNLW} with $p=7$ ($s_c=7/6$). Like our $u_{X,\beta(X)}$ and $u_{X, -\beta(-X)}$, these solutions are global and have an asymptotic self-similar behaviour inside the wave cone. However, the initial data of these solutions are in $\mathcal{H}^{s}$ for $s>7/6$, but neither in the critical space $\mathcal{H}^{7/6}$ nor in the energy space $\mathcal{H}^1$.
\end{itemize}

We next give precise asymptotics of the $L^p$ and $\mathcal{H}^s$ norms of the threshold solutions. We will use the following notation: we denote $a(t)=O(b(t))$ if there exists a constant $C$ independent of $t$ such that $|a(t)|\leq Cb(t)$ for large $t$, and $a(t)\approx b(t)$ if $a(t)=O(b(t))$ and $b(t)=O(a(t))$.
\begin{theo}[Asymptotics of Lebesgue and Sobolev norms at the threshold]
\label{thm:self2}
Let $(X,Y)\in \mathbb{R}^2$ with $Y=\beta(X)$ or $Y=-\beta(-X)$. Denote $u=u_{X,Y}$. Then we have the following asymptotics for large $t$. 
If $p>\frac 32$,
\begin{gather}
\label{eq:asymptoticLp}
    \frac{1}{(4\pi)^{1/p}}\|u(t)\|_{L^p(\R^3)}=\frac{\sqrt{2}}{3^{\frac 1p}}t^{\frac{3}{p}-1}+O\left(t^{\frac 2p-1}\right),\\
    \label{eq:boundLp}
    \forall M>0,\quad \|u(t)\|_{L^p(\{|x|>t-M\})}=O\left(t^{\frac{2}{p}-1}\right),
\end{gather}
while if $p\geq 1$,
\begin{equation}
    \label{eq:asymptoticdtLp}
     \|\partial_tu(t)\|_{L^p} \approx t^{-1+\frac{2}{p}}.
\end{equation}
On the other hand, if $0\leq \nu <1/2$,
\begin{equation}
\label{eq:asymptoticHs}
\|u(t)\|_{\dot{H}^\nu}^2=\kappa_{\nu}t^{1-2\nu}+O\left(t^{\frac 12-\nu}\right),\quad \|\partial_tu(t)\|_{\dot{H}^{\nu-1}}^2=O\left(t^{1-2\nu}\right),
\end{equation}
where $\kappa_{\nu}=128\pi^3 \int_0^{\infty} (\sin \sigma-\sigma \cos \sigma)^2\sigma^{-4+2s}d\sigma$, and moreover
\begin{equation}
    \label{eq:asymptoticH1/2}
\|u(t)\|_{\dot{H}^{1/2}}^2=64\pi^3\log t+O(\sqrt{\log t}),\quad  \|\partial_tu(t)\|_{\dot{H}^{-1/2}}^2=O(\log t).
\end{equation}
 Finally, if $\nu>1/2$ 
 \begin{equation}
     \label{eq:asymptoticHs'}
     \|u(t)\|_{\dot{H}^\nu}^2\approx 1,\quad \|\partial_tu(t)\|_{\dot{H}^{\nu-1}}^2\approx 1.
\end{equation}
\end{theo}
\begin{remark}
Let $\nu\in [0,3/2)$, and $p=\frac{6}{3-2\nu}$ the unique Lebesgue exponent such that $ \dot{H}^{\nu}\subset L^p$. Then, as $t\to\infty$, $\|u(t)\|_{\dot{H}^{\nu}}$ and $\|u(t)\|_{L^p}$ have the same (diverging) behaviour at infinity for $\nu\in [0,1/2)$. This breaks down at $\nu=1/2$: $\|u(t)\|_{L^3}$ is bounded, whereas the corresponding Sobolev norm $\|u(t)\|_{\dot{H}^{1/2}}$ goes to infinity. For $\nu\in (1/2,3/2)$, the corresponding $L^p$ norm goes to $0$ polynomially (a typical dispersive behaviour), which is not the case of the $\dot{H}^{\nu}$ norm of $u$.  
\end{remark}

We recall here that by \cite{DodsonLawrie15}, if $u$ is a radial solution of \eqref{NLW} such that $\boldsymbol{u}$ is bounded in $\mathcal{H}^{1/2}$, then it scatters to a linear solution. 
Theorem \ref{thm:self2} proves that this result is almost optimal, since a logarithmic growth at infinity of this norm is possible for a non-scattering solution, for which all the higher order norms are bounded.  We believe that solutions with an asymptotic self-similar behaviour in the wave cone will always exhibit these two features. We thus conjecture, in accordance with Conjecture \ref{conj:global}:
\begin{conj}
Let $u$ be a solution of \eqref{NLW}, defined for $t\in [0,\infty)$ with $(u_0,u_1)\in \mathcal{H}^{1/2}\cap \mathcal{H}^1$. Then for all $\nu\in (1/2,1]$,
$$\limsup_{t\to\infty}\|\boldsymbol{u}(t)\|_{\mathcal{H}^{\nu}}<\infty.$$
\end{conj}
\begin{conj}
There exists $\beta>0$ such that for any  solution $u$ of \eqref{NLW}, defined on the maximal time interval $t\in [0,t_+)$, with initial data in $\mathcal{H}^{1/2}$, if
$$ \sup_{2\le t<t_+}\|\boldsymbol{u}(t)\|_{\mathcal{H}^{1/2}}(\log t)^{-\beta}<\infty$$
then $u$ is global ($t_+=\infty$) and scatters to a linear solution.
\end{conj}

This conjecture is related to the work \cite{MeRa08}, where a logarithmic lower bound of the critical norm was obtained for finite time blow-up for the intercritical non-linear Schr\"odinger equation. We do not know of any analog of \cite{MeRa08} for the wave equation \eqref{gNLW}. We refer to \cite{KiStVi14} for some estimates on blow-up for wave equations.

\subsection{Blow-up solutions}
We finally discuss the solutions $u_{X,Y}$ in the (Blow-up) case of
Theorem~\ref{thm:main}. 
\begin{theo}[Asymptotics of Lebesgue and Sobolev norms for blow-up]\label{thm:blowup}
Consider $(X, Y)\in \mathbb R^2$ such that either $Y>\beta(X)$ or $Y<-\beta(-X)$, so that $t_+(X, Y)<\infty$ by Theorem~\ref{thm:main}. Then, as $t\nearrow t_+=t_+(X, Y)$, we have the asymptotics
\begin{equation}\label{eq:blowup_Lthree}
    \displaystyle \lVert u(t)\rVert_{L^3(\mathbb R^3)}= \frac{C_0^\frac13}{(t_+-t)^\frac{1}{2}}+O(1);
\end{equation}
\begin{equation}\label{eq:blowup_sobolev}
    \lVert u(t)\rVert_{\dot{H}^{1/2}} \approx \frac{1}{(t_+-t)^\frac{1}{2}};
\end{equation}
\begin{equation}\label{eq:blowup_ut_Lthreetwo}
   \lVert \partial_t u(t)\rVert_{L^\frac{3}{2}(\mathbb R^3)}= \frac{2^{-\frac12}C_0^\frac23}{t_+-t}+O(1);
\end{equation}
\begin{equation}\label{eq:blowup_ut_sobolev}
    \lVert \partial_t u(t)\rVert_{\dot{H}^{-1/2}(\mathbb R^3)}= O\left(\frac{1}{t_+-t}\right).
\end{equation}
The constant $C_0$ depends on the blow-up time $t_+$ only, and is given by
\begin{equation}\label{eq:C_Zero_Constant}
C_0=2^\frac72 \pi \int_0^\infty\!\! \left(1+\frac{\rho^2t_+}{1+t_+^2}\right)^{-3}\rho^2\, d\rho.
\end{equation}
\end{theo}
We conjecture that our blow-up solutions are stable in the critical space $\mathcal{H}^{1/2}$. That is, if $v$ is a solution of \eqref{NLW} with $\|\boldsymbol{u}_{X,Y}(0)-\boldsymbol{v}(0\|_{\mathcal{H}^{1/2}}$ sufficiently small, then we conjecture that $v$ blows-up in finite time with the same type of blow-up as $u_{X,Y}$. See \cite{DonningerSchorkhuber12, DonSch17} for results of this type. Using only elementary comparison arguments we can prove the following weaker result. We denote by $H^s=H^s(\mathbb R^3)$ the usual inhomogeneous $L^2$-Sobolev space. 
\begin{theo}[Stability of blow-up] \label{thm:blowup_stable}
 Let $u=u_{X,Y}$ be one of the \emph{(Blow-up)} solutions of Theorem~\ref{thm:main} and let $t_+=t_+(X,Y)$ be its finite maximal time of existence. Let $\eta>0$. Then there exists $\varepsilon=\varepsilon(u, \eta)>0$ such that for all radial solution $v$ of \eqref{NLW} with $\left\|\boldsymbol{u}(0)-\boldsymbol{v}(0)\right\|_{H^3\times H^2} <\varepsilon$, $v$ blows up in $H^3\times H^2$ at finite positive time $t_+'<t_++\eta$.  
\end{theo}
We will also study the behaviour of our blow-up solutions at points of the blow-up surface. This kind of study of the \emph{blow-up mechanism} for nonlinear wave equations has been initiated by Merle and Zaag in a series of papers (see~\cite{MeZa15} and references therein). We will first of all establish the maximal domain of definition of the blow-up solutions, observing that all points on the boundary of such domain are \emph{non-characteristic} (see the forthcoming Remark~\ref{rem:characteristic_points}). Then, in the subsection~\ref{sec:blowup_profile}, we will study our solutions in self-similar coordinates. We will find that they have the same profile as the one predicted by Merle--Zaag in~\cite[p.~3]{MeZa15}.



We conclude with the following proposition, showing that all our blow-up solutions do converge to an attractor in accordance with Bizo\'n--Zengino\v{g}lu~\cite[Conjecture~3]{BizZen09}.
\begin{prop}[Convergence to an attractor]\label{prop:blow_up_attractor}
    Let $u=u_{X,Y}$ be one of the \emph{(Blow-up)} solutions of Theorem~\ref{thm:main} and let $t_+=t_+(X,Y)$ be its finite maximal time of existence. Then letting $T_+=2\arctan(t_+)$ we have as $t\nearrow t_+$ 
    \begin{equation}\label{eq:blow_up_attractor}
        u(t, r)=\frac{2\sqrt 2}{a(1+r^2-t^2)-2bt}+ O(t_+-t)
    \end{equation}
    for $a=\sin T_+, b=\cos T_+$.
\end{prop}

The paper is organized as follows. In Section~\ref{sec:ODE} we will reduce the cubic wave equation~\eqref{NLW} with the specific initial data~\eqref{eq:uAB_initial} to an ODE which we will study in detail. In Section~\ref{sec:threshold} we will prove Theorems~\ref{thm:self0}--\ref{thm:self2}, which contain all our estimates on the threshold solutions. Finally, in Section~\ref{sec:blowup_profile} we will prove Theorem~\ref{thm:blowup}, Theorem~\ref{thm:blowup_stable} and Proposition~\ref{prop:blow_up_attractor} concerning our blow-up solutions,  as well as discussing their blow-up mechanism in the sense of Merle and Zaag.

\section{Dynamics of the Duffing ordinary differential equation}\label{sec:ODE}
In this section we prove Theorem~\ref{thm:main}. The main ingredient is the analysis of the ODE initial value problem
\begin{equation}\label{eq:DuffingIVP}
    \begin{cases}
        \ddot{U}(s) + U(s) = U^3(s), \\
        U(0)=X, \dot{U}(0)=Y,
    \end{cases}
\end{equation}
known as \emph{undamped softening Duffing equation}. The relationship between~\eqref{eq:DuffingIVP} and the cubic wave equation~\eqref{NLW} is explained by the following lemma, which relies on the classical Penrose compactification of the Minkowski spacetime $\mathbb R^{1+3}$. Recall that the cotangent is defined as $\cot(s)~=~1/\tan(s)$, with $\cot(\pm\pi/2)=0$.
\begin{lemma}\label{lem:penrose}
    Suppose that $U=U(s)$ is a smooth solution to~\eqref{eq:DuffingIVP}, defined for $s\in(T_-, T_+)$, where $T_-\in[-\infty, 0)$ and $T_+\in(0, \infty]$. Then the function 
    \begin{equation}\label{eq:u_from_U}
        u(t, x)=\frac{2 U\big(\arctan(t+\lvert x \rvert)+\arctan(t-\lvert x\rvert)\big)}{\sqrt{1+(t-\lvert x\rvert)^2}\sqrt{1+(t+\lvert x \rvert)^2}}
    \end{equation}
    is defined for all $(t, x)\in\mathbb R^{1+3}$ such that 
    \begin{equation}\label{eq:t_r_range}
        M_-(T_-, \lvert x\rvert) < t < M_+(T_+, \lvert x \rvert),
    \end{equation}
    where 
    \begin{equation}\label{eq:M_functions}
        M_\pm(T, r):=
        \begin{cases}
            -\cot(T)\pm\sqrt{1+r^2+\cot^2(T)}, &\lvert T\rvert <\pi, \\
            \pm \infty, & \lvert T\rvert\ge \pi;
        \end{cases}
    \end{equation}
    (see Figure~\ref{fig:hyperboloids}). Also, $u$ is smooth and it satisfies 
    \begin{equation}\label{eq:NLW_recall}
        \begin{cases}
            \partial_t^2 u(t, x)-\Delta u(t, x)=u^3(t, x), \\
            u(0, x)=\frac{2X}{1+\lvert x\rvert^2},\ \partial_t u(0, x)=\frac{4Y}{(1+\lvert x \rvert^2)^2} 
        \end{cases}
    \end{equation}
    at all points of its domain of definition.
\end{lemma}
\begin{figure}
    \centering
    \begin{tikzpicture}
        \begin{axis}[axis on top, axis x line=middle, axis y line=middle, xmin=-0.5, xmax=4, ymin=-2.3, ymax=3.34, xlabel=$\scriptsize{r=\lvert x \rvert}$, ylabel=$\scriptsize{t}$, xtick=\empty, ytick=\empty]
            \addplot+[mark=none, domain=0:3, draw=none, name path=up]{0.05+sqrt(1+9+(0.05)^2)};
            \addplot+[mark=none, draw=blue, domain=0:3, name path=down] {0.05+sqrt(1+x^2+(0.05)^2)};
            \addplot[lightgray] fill between[of=up and down];

            \addplot+[mark=none, draw=none, domain=0:3, name path=upp]{1-sqrt(2+9)};
            \addplot+[mark=none, draw=blue,  name path=downp, domain=0:3] {1-sqrt(2+x^2)};
            \addplot[lightgray] fill between[of=upp and downp];
            \draw[blue] (2.49, 2.6) node[anchor=west]{$t=M_+(T_+, r)$};
            \draw[blue] (2.49, 1-2.6) node[anchor=west]{$t=M_-(T_-, r)$};
        \end{axis}
    \end{tikzpicture}
    \caption{If $\lvert T_\pm\rvert<\pi$, then $u=u(t, \lvert x \rvert)$ is defined in the unshaded region, between the two space-time hyperboloids.}
    \label{fig:hyperboloids}
\end{figure}
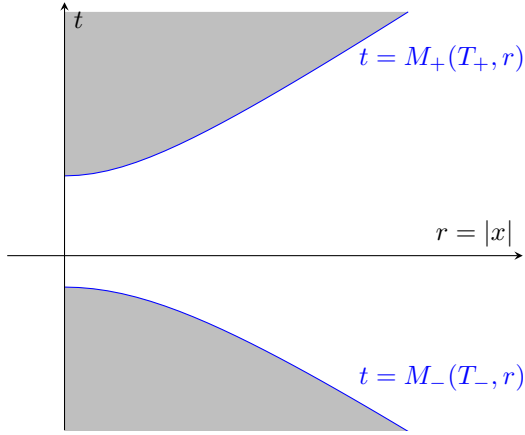
\begin{remark}\label{rem:characteristic_points}
    The domain~\eqref{eq:t_r_range} is called in the literature the \emph{maximal influence domain} of $u$; see, for example,~\cite{Alinhac95Bo}. In the blow-up case $\lvert T_\pm\rvert <\pi$, we see that $\lvert \partial_r M_{\pm}(T_\pm, r)\rvert<1$, for all $r>0$. Thus every blow-up point is non-characteristic, in the sense of Merle and Zaag~\cite{MeZa08}. 
\end{remark}
\begin{proof}
    We denote the generic point of $\mathbb S^3$ by the coordinates
    \begin{equation}\label{eq:generic_sphere}
        \begin{array}{cc}
            (\cos \theta, \omega\sin \theta ), &\text{where }\theta\in [0, \pi],\omega\in\mathbb S^2.
        \end{array}
    \end{equation}
    We recall that a function $V=V(s, \cos \theta , \omega \sin \theta)$ solves the cubic wave equation $\partial_s^2V -\Delta_{\mathbb S^3} V + V = V^3$ on an open subset of $\mathbb R\times \mathbb S^3$ if and only if the function $v$ defined by
    \begin{equation}\label{eq:Penrose_transform_functions}
        v(t, x)=\Omega(t, \lvert x\rvert)V(s, \cos\theta, \omega\sin \theta),
    \end{equation}
    where $s=\arctan(t+\lvert x\rvert) + \arctan(t-\lvert x \rvert), \theta=\arctan(t+\lvert x\rvert) - \arctan(t-\lvert x \rvert)$ and $\omega= x/\lvert x\rvert$, solves the partial differential equation in~\eqref{eq:NLW_recall} on the corresponding open subset of $\mathbb R^{1+3}$. This follows at once from the conformality of the map given by $(t, x)\mapsto (s, \cos \theta, \omega \sin \theta)$, where $(s, \cos \theta, \omega \sin \theta)$ depends on $(t, x)$ via the equations in~\eqref{eq:Penrose_transform_functions}\footnote{See~\cite[Appendix~A.4]{Hor97}. Alternatively, see~\cite[Lemma~A.1]{NeDoSStTa23} for a source that does not rely on tools from conformal geometry.}.
    
    In the special case $V(s, \cos\theta, \omega \sin\theta)=U(s)$, where $U$ is a solution to the ODE~\eqref{eq:DuffingIVP}, the formula~\eqref{eq:Penrose_transform_functions} reduces exactly to~\eqref{eq:u_from_U}. Thus, $u=u(t, x)$ defined by~\eqref{eq:u_from_U} solves the cubic wave equation~\eqref{eq:NLW_recall}, and the initial conditions follow from direct computation.
    
    It remains to determine the domain of this function $u$. We claim that the strip $s\in (T_-, T_+)$, which is the domain of $U=U(s)$, is mapped onto the region~\eqref{eq:t_r_range}. This claim follows from the formulas 
    \begin{equation}\label{eq:inverse_Penrose}
        \begin{array}{cc}
            \sin s= \Omega t, & \cos s= \Omega \frac{1+\lvert x\rvert^2-t^2}{2}, 
        \end{array}
    \end{equation}
    which are easily obtained by inverting~\eqref{eq:Penrose_transform_functions} (or see~\cite[pg.~277]{Hor97}). These formulas imply the quadratic equation $1+\lvert x \rvert^2 - t^2 = 2t \cot(s)$, which can be solved for $t$ to obtain the two functions $M_\pm$. This proves the claim and concludes the proof of the lemma.
\end{proof}

By standard ODE theory, for each $(X, Y)\in\mathbb R^2$ the unique solution $U=U_{X, Y}(s)$ to the initial value problem~\eqref{eq:DuffingIVP} is defined in a maximal interval $s\in (T_-(X, Y), T_{+}(X, Y))$, where $T_\pm$ are smooth functions on $\mathbb R^2$ which we will determine explicitly in the next subsection. The function $u=u_{X, Y}(t, x)$ corresponding to $U_{X,Y}$ via~\eqref{eq:u_from_U} exists for all $t>0$ (resp.~$t<0$) if and only if $T_+\ge \pi$ (resp. $T_-\le- \pi$). Otherwise, $u_{X, Y}$ blows up in finite future time at 
\begin{equation}\label{eq:blow_up_forward}
    \begin{array}{cc}
        x=0, & t\nearrow M_+(T_+, 0)=\sqrt{1+\cot^2(T_+)}-\cot(T_+)=\frac{1-\cos T_+}{\sin T_+}, 
    \end{array}
\end{equation}
(resp. in finite past time at $x=0, t\searrow  -\sqrt{1+\cot^2(T_-)}-\cot(T_-)$). We thus see that the cases $(i)$=(Blow-up), $(ii)$=(Scattering) and $(iii)$=(Threshold) of Theorem~\ref{thm:main} correspond to the level sets 
\begin{equation}\label{eq:level_sets_Tplus}
    \begin{array}{ccc}
        (i)\ T_+(X, Y)< \pi, & (ii)\ T_+(X, Y)>\pi, & (iii)\ T_+(X, Y)=\pi. 
    \end{array}
\end{equation}
which we will determine explicitly in subsection~\ref{sec:subsection_beta}. We will use the following ODE energy conservation:
\begin{equation}\label{eq:ODEenergy_recall}
    E_{X, Y}=\frac{Y^2}{2}+\frac{X^2}{2}-\frac{X^4}{4}=\frac{\dot{U}_{X, Y}(s)^2}{2}+\frac{U_{X,Y}(s)^2}{2}-\frac{U_{X,Y}(s)^4}{4}. 
\end{equation}
We pause for a moment to record the following identities, which relate precisely the energy and norm of $\boldsymbol{u}_{X, Y}$ and the corresponding $U_{X, Y}$. These will not be needed in the following but are interesting on their own.
\begin{prop}\label{prop:energy_penrose}
    The following relations hold:
    \begin{equation}\label{eq:energy_norm_penrose_physical}
        \begin{split}
            E_{X, Y}&=\frac{1}{\lvert\mathbb S^3\rvert} \int_{\mathbb R^3}  \frac{ (\partial_t u_{X, Y}(0, x))^2}{2} + \frac{\lvert \nabla u_{X, Y}(0, x)\rvert^2}{2} -\frac{(u_{X, Y}(0, x))^4}{4}\, dx, \\
            X^2 + Y^2 &= \frac{1}{\lvert \mathbb S^3\rvert}\lVert \boldsymbol{u}_{X, Y}(0)\rVert_{\mathcal H^{1/2}}^2;
        \end{split}
    \end{equation}
    recall $\lvert \mathbb S^3\rvert= 2\pi^2$. 
\end{prop}
\begin{proof}
    By~\eqref{eq:Penrose_transform_functions}, $u_{X, Y}(t, x)=\Omega(t, \lvert x \rvert)U_{X, Y}(s)=(\cos s + \cos \theta)U_{X, Y}(s)$. A computation reveals that 
    \begin{equation}\label{eq:partial_t_u_zero}
        \partial_t u_{X,Y}(0, x)=(1+\cos \theta)^2\dot{U}_{X, Y}(0).
    \end{equation}
    Now we write $\int_{\mathbb R^3} \lvert \nabla u_{X, Y}\rvert^2=\int_{\mathbb R^3}\lvert \sqrt{-\Delta} u_{X, Y}\rvert^2$ and we invoke the following \emph{intertwining law} (see~\cite[Lemma~A.3]{NeDoSStTa23}):
    \begin{equation}\label{eq:intertwining_law}
        \sqrt{-\Delta}u_{X, Y}(0, x)=(1+\cos \theta)^2\sqrt{1-\Delta_{\mathbb S^3}}(U_{X, Y}(0))=(1+\cos \theta)^2U_{X, Y}(0),
    \end{equation}
    where we used that $\sqrt{1-\Delta_{\mathbb S^3}}(U(0))=U(0)$, since $U(0)$ is constant on $\mathbb S^3$. 
    
    Denoting by $d\sigma$ the surface measure on $\mathbb S^3$, and applying the change of variables~\eqref{eq:Penrose_transform_functions} at $t=0$, we find the Jacobian $dx=(1+\cos \theta)^{-3}d\sigma$. We conclude 
    \begin{equation}\label{eq:compute_energy}
        \begin{split}
        \int_{\mathbb R^3}\!\!  \frac{ (\partial_t u_{X, Y}(0, x))^2}{2} + \frac{\lvert \nabla u_{X, Y}(0, x)\rvert^2}{2} -\frac{(u_{X, Y}(0, x))^4}{4}\, dx&=E_{X, Y}\!\int_{\mathbb S^3}(1+\cos \theta)\, d\sigma\\ 
        &=\lvert \mathbb S^3\rvert E_{X, Y},
        \end{split}
    \end{equation}
    proving the first identity in~\eqref{eq:energy_norm_penrose_physical}. To prove the second identity, start by noting
    \begin{equation}\label{eq:norm_Hone_half_explicit}
        \lVert \boldsymbol{u}_{X, Y}(0)\rVert_{\mathcal{H}^{1/2}}^2=\int_{\mathbb R^3} \left( u_{X, Y}\sqrt{-\Delta}u_{X,Y}(0, x) + \partial_t u_{X, Y}\sqrt{-\Delta}^{-1}\partial_t u_{X,Y}(0, x)\right)\,dx.
    \end{equation}
    Now combining~\eqref{eq:partial_t_u_zero} and~\eqref{eq:intertwining_law} yields
    \begin{equation}\label{eq:intertwining_law_negative}
        \sqrt{-\Delta}^{-1}\partial_t u_{X, Y}(0, x)=(1+\cos \theta)^{-1}\sqrt{1-\Delta_{\mathbb S^3}}^{-1}(\dot{U}_{X, Y}(0))=(1+\cos \theta)^{-1}\dot{U}_{X, Y}(0).
    \end{equation}
    Using this together with~\eqref{eq:intertwining_law} we obtain 
    \begin{equation}\label{eq:end_norm_proof}
        \lVert\boldsymbol{u}_{X, Y}(0)\rVert^2_{\mathcal{H}^{1/2}}=\left(U_{X, Y}(0)^2+\dot{U}_{X, Y}(0)^2\right)\int_{\mathbb S^3}\, d\sigma=\lvert\mathbb S^3\rvert (X^2+Y^2), 
    \end{equation}
    concluding the proof.
\end{proof}
\subsection{Determining the functions $T_\pm$}
In the cases $E_{X, Y}=1/4$ or $E_{X, Y}=0$, the ODE energy conservation law~\eqref{eq:ODEenergy_recall} can be explicitly integrated in terms of trigonometric and hyperbolic functions. Recalling that 
\begin{equation}\label{eq:hyperbolic_tangent}
    \begin{array}{cc}
        \tanh(x)=\frac{e^x-e^{-x}}{e^x+e^{-x}}, & \arctanh(x)=\frac12\log\frac{1+x}{1-x}, 
    \end{array}
\end{equation}
we have, for $E_{X, Y}=1/4$, the one-parameter family of solutions
\begin{equation}\label{eq:E_One_Fourth}
    U_{X, \pm\frac{\lvert X^2 -1\rvert}{2}}(s)=
        \begin{cases}
                    \tanh\left[\pm\frac{s}{\sqrt 2} + \arctanh(X)\right], & \lvert X\rvert <1, \\ 
                    \left(\tanh\left[\arctanh\left(\frac1X\right)\mp \frac{ s}{\sqrt 2 } \right]\right)^{-1}, & \lvert X\rvert >1,
        \end{cases}
\end{equation}
as well as the constant solutions $U_{\pm 1, 0}(s)\equiv\pm 1$. For $E_{X, Y}=0$ we have the null solution $U_{0, 0}\equiv 0$ and the one-parameter family, for $\lvert X\rvert>\sqrt{2}$,
\begin{equation}\label{eq:E_Zero}
    U_{X, \pm\sqrt{\frac{X^4}{2}-X^2}}(s)=\pm \frac{\sqrt 2}{\sin\left(\arcsin\left(\frac{\sqrt{2}}{X}\right)-s\right)}.
\end{equation}

For all these explicit solutions, computing the maximal times of existence $T_\pm$ will be immediate. On the other hand, when $E_{X, Y}\notin \{0, \tfrac14\}$ we will not have such explicit formulas, and instead we will express $T_\pm$ in terms of the following integrals:
\begin{equation}\label{eq:R_S_integrals_recall}
    \begin{split}
        R(X, Y)&=\int_{X\sign(Y)}^\infty\frac{dv}{\sqrt{2E_{X, Y}-v^2+\frac12 v^4}},\ (\text{convention:  }X\sign(Y)|_{Y=0}=\lvert X\rvert) \\ S(X, Y)&=\left(\int_{\sqrt{1+\sqrt{1-4E_{X,Y}}}}^{\lvert X\rvert} + \int_{\sqrt{1+\sqrt{1-4E_{X,Y}}}}^\infty \right)\frac{dv}{\sqrt{2E_{X, Y}-v^2+\frac12 v^4}}.
    \end{split}
\end{equation}
The reason for the convention in the first formula will be apparent in the proof of the following lemma. We also agree that these functions equal $+\infty$ at those $(X, Y)$ for which one of the integrals is not convergent or one of the square roots has a negative argument; see the forthcoming Remark~\ref{rem:R_S_domain} for details.
\begin{lemma}\label{lem:T_Plus_Expression}
    For $(X, Y)\in\mathbb R^2$, let $U=U_{X, Y}(s)$ be the unique solution to~\eqref{eq:DuffingIVP} and let $T_+(X, Y)$ be its maximal positive time of existence.
    \begin{itemize}
        \item[(i)] If $E_{X, Y}>1/4$, then $T_+(X, Y)=R(X, Y)$. 
        \item[(ii)] If $E_{X, Y}<1/4$, then 
        \begin{equation}\label{eq:T_plus_below}
        T_+(X, Y)=
        \begin{cases}
                \infty, & \lvert X\rvert \le 1,\\
                R(X, Y), & \lvert X\rvert >1, XY\ge 0, \\
                S(X, Y), & \lvert X\rvert>1, XY<0.
            \end{cases}
        \end{equation}
    \end{itemize}
    In all cases, the maximal negative time of existence is $T_-(X, Y)=-T_+(X, -Y)$.
\end{lemma}
\begin{remark}\label{rem:energy_onefourth}
    In the case $E_{X, Y}=1/4$ we have a more explicit expression, immediate consequence of~\eqref{eq:E_One_Fourth}:
    \begin{equation}\label{eq:T_plus_energy_critical}
            T_+(X, Y)=
            \begin{cases} 
                \infty, & \lvert X\rvert \le 1, \\
                \sqrt 2 \arctanh\left(\frac1{\lvert X\rvert}\right), & \lvert X\rvert >1, XY>0, \\ 
            \infty, & \lvert X\rvert >1, XY\le 0.
            \end{cases}
        \end{equation}
\end{remark}
Inspecting the formulas of the previous lemma yields the following corollary, which we will use in the next subsection.
\begin{corol}\label{cor:Tplus_Plus_Tminus}
For all $(X, Y)\in \mathbb R^2$, 
\begin{equation}\label{eq:Tplus_Plus_Tminus}
    T_+(X, Y)+\lvert T_-(X, Y)\rvert = \begin{cases}\displaystyle
        2\int_0^\infty \frac{dv}{\sqrt{2E_{X, Y}-v^2+\frac12 v^4}}, & E_{X, Y}>\frac14, \\ \infty, & E_{X, Y}=\frac14,\\\displaystyle
        2\int_{\sqrt{1+\sqrt{1-4E_{X, Y}}}}^\infty \frac{dv}{\sqrt{2E_{X, Y}-v^2+\frac12 v^4}}, & E_{X, Y}<\frac14.
    \end{cases}
\end{equation}
\end{corol}
    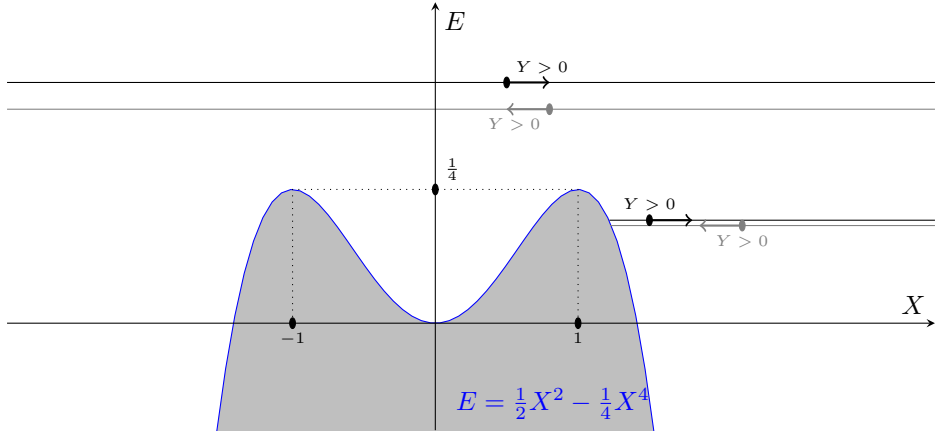
\begin{figure}
        \centering
        \begin{tikzpicture}
            %
            %
            %
            \begin{axis}[axis on top, axis x line=middle, axis y line=middle, 
            xmin=-3, xmax=3.5, ymin=-0.4, ymax=1.2, 
            xlabel=$\scriptsize{X}$, ylabel=$\scriptsize{E}$, 
            ytick=\empty, xtick=\empty, 
            xscale=1.8, 
            samples=50]
                %
                %
                %
                \addplot+[mark=none, domain=-1.6:1.6, draw=none]{-0.7168} \closedcycle;
                \addplot+[mark=none, fill=lightgray, draw=blue, domain=-1.6:1.6] {x^2-0.5*x^4};
                \draw[dotted] (-1,0)  node[anchor=north]{\tiny $-1$}--(-1, 0.5)--(1,0.5)--(1,0) node[anchor=north]{\tiny $1$};
                \draw[fill] (-1,0) circle(0.02);
                \draw[fill] (1,0) circle(0.02);
                \draw[fill] (0, 0.5) circle(0.02) node[anchor=south west]{\tiny $\tfrac14$};
                \draw (-3, 0.9)--(3.5, 0.9);
                \draw[fill] (0.5, 0.9) circle(0.02) node[above right=0pt]{\tiny $Y>0$};
                \draw[thick, ->] (0.5, 0.9)--(0.8, 0.9);
                \draw[gray] (-3, 0.8)--(3.5, 0.8);
                \draw[gray, fill] (0.8, 0.8) circle(0.02) node[below left=0pt]{\tiny $Y>0$};
                \draw[thick, gray, ->] (0.8, 0.8)--(0.5, 0.8);
                \draw (1.22, 0.385) -- (3.5, 0.385);
                \draw[gray] (1.224, 0.365) -- (3.5, 0.365);
                \draw[fill] (1.5, 0.385) circle(0.02) node[above]{\tiny $Y>0$};
                \draw[thick, ->] (1.5, 0.385)--(1.8, 0.385);
                \draw[fill, gray] (2.15, 0.365) circle(0.02) node[below]{\tiny $Y>0$};
                \draw[thick, ->, gray] (2.15, 0.365)--(1.85, 0.365);
                
                \draw[blue] (1.57,-0.3) node[left]{$E=\tfrac12 X^2-\tfrac14 X^4$};
            \end{axis}
        \end{tikzpicture}
        \caption{Solutions to the ODE initial value problem~\eqref{eq:DuffingIVP} in the $(X, E)$ plane. For $E<1/4$, the solutions may ``bounce" on the boundary of $\Ran(\Phi)$, given by $E=\frac12 X^2-\frac14 X^4$.} 
        \label{fig:Pos_Energy_Plane}
    \end{figure}
\begin{proof}[Proof of Lemma~\ref{lem:T_Plus_Expression}]
     If $\dot{U}_{X, Y}(s)$ does not change sign on an interval $(s_0, s_1)$, then we can solve for $\dot{U}(s)$ in the energy conservation law~\eqref{eq:ODEenergy_recall} and integrate, obtaining
    \begin{equation}\label{eq:Integrate_Energy_General}
        s_1-s_0=\sign(\dot U(s))\int_{U_{X, Y}(s_0)}^{U_{X, Y}(s_1)} \frac{dv}{\sqrt{2E_{X, Y} - v^2+\frac12 v^4}}.
    \end{equation}
    In order to use this, we consider the map $\Phi\colon (X, Y)\to (X, E_{X, Y})$. This map is 2:1, since $\Phi(X_1, Y_1)=\Phi(X_2, Y_2)$ if and only if $X_1=X_2$ and $\lvert Y_1\rvert= \lvert Y_2\rvert$. Its range is 
    \begin{equation}\label{eq:range_phi}
        \Ran(\Phi)=\{(X, E)\ :\ E\ge \frac12 X^2-\frac14 X^4\}.
    \end{equation}
    We will treat $\Phi$ as a change of variable in the phase space $(X, Y)$, and we will study the function $T_+$ in the plane $(X, E)$, where the trajectories of the solutions to~\eqref{eq:DuffingIVP} are especially simple; indeed, if $U=U(s)$ solves~\eqref{eq:DuffingIVP}, then $\Phi(U(s), \dot{U}(s))$ traces a horizontal line to the left or to the right depending on the sign of $\dot{U}(s)$, see Figure~\ref{fig:Pos_Energy_Plane}. The boundary of $\Ran(\Phi)$ contains the points at which $\dot{U}(s)=0$. 
    
    In the case $E_{X, Y}>1/4$, the line $(U_{X, Y}(s), E_{X, Y})$ never touches the boundary of $\Ran(\Phi)$, so $\sign(\dot{U}_{X, Y}(s))=\sign(Y)$ for all $s$ and~\eqref{eq:Integrate_Energy_General} yields
    \begin{equation}\label{eq:IntegrateEnergyEasy}
        \int_{X}^{U_{X, Y}(s)}\frac{ \sign(Y)dv}{\sqrt{2E_{X, Y} - v^2 + \frac12 v^4}} = s.
    \end{equation}
    The denominator never vanishes, so the integral is absolutely convergent on $(-\infty, \infty)$. This implies that $U_{X, Y}(s)$ blows up in finite time at $\infty$ if $Y>0$, or at $-\infty$ if $Y<0$. So, letting $s\to T_+(X, Y)$ in~\eqref{eq:IntegrateEnergyEasy} and changing variable in the integral in the case $Y<0$, we prove~(i).
    
    In the case $E_{X, Y}<1/4$, we see that those solutions with $\lvert X\rvert<1$ are bounded for all times, as they are trapped in a potential well\footnote{We will give an explicit expression for these solutions at the end of this subsection.}. So in particular $T_+(X,Y)=\infty$. On the other hand, for $\lvert X\rvert>1$, we have $\lvert U_{X, Y}(s)\rvert\ge \sqrt{1+\sqrt{1-4E_{X, Y}}}>1$ for all $s$. Since $\ddot{U}(s)$ satisfies
    \begin{equation}\label{eq:ODE_convexity}
        \ddot{U}(s)=U^3(s)-U(s),
    \end{equation}
    it is always strictly positive (if $X>1$) or strictly negative (if $X<-1$). So, if $XY>0$, then $\dot{U}_{X, Y}(s)$ never changes sign for $s>0$, which means that the line $\Phi(U_{X, Y}(x), \dot{U}_{X, Y}(s))$ does not touch the boundary of $\Ran(\Phi)$ and the analysis of the previous case applies, yielding $T_+=R$. We remark that this also holds if $Y=0$, in which case $\lvert X\rvert=\sqrt{1+\sqrt{1-4E_{X, Y}}}$; this is the reason of our choice of conventions in the formula~\eqref{eq:R_S_integrals_recall} for $R(X, Y)$. 
    
    If, on the other hand, $XY<0$ and $E_{X, Y}<1/4$, then $\lvert U(s)\rvert$ reaches its global minimum $\sqrt{1+\sqrt{1-4E_{X, Y}}}$ at the time 
    \begin{equation}\label{eq:speed_change_sign}
        s=\int_{\sqrt{1+\sqrt{1-4E_{X, Y}}}}^X \frac{dv}{\sqrt{2E_{X, Y} -v^2+ \frac12 v^4}}, 
    \end{equation}
    as we compute by applying~\eqref{eq:Integrate_Energy_General}; at this point, the line $\Phi(U(s), \dot{U}(s))$ touches the boundary of $\Ran(\Phi)$ and $\dot{U}$ changes sign. Applying~\eqref{eq:Integrate_Energy_General} again we get
    \begin{equation}\label{eq:T_Plus_Bounce}
        T_+(X, Y)=\left(\int_{\sqrt{1+\sqrt{1-4E_{X, Y}}}}^X + \int_{\sqrt{1+\sqrt{1-4E_{X, Y}}}}^\infty \right)\frac{dv}{\sqrt{2E_{X, Y} -v^2+ \frac12 v^4}}=S(X, Y).
    \end{equation}
    We have thus proved the point~(ii). 
    
    To conclude the proof, it suffices to note that $U_{X, Y}(s)=U_{X, -Y}(-s)$, hence $T_-(X, Y)=-T_+(X, -Y)$, as claimed.
\end{proof}
\begin{remark}\label{rem:R_S_domain}
    We can now give the announced description of the domain of definition of $R$ and of $S$. The analysis of the previous proof shows that $R(X, Y)<\infty$ in the region 
    \begin{equation}\label{eq:Dom_R}
        \begin{split}
            \operatorname*{Dom}(R)=&\{(X, Y)\in\mathbb R^2\ : E_{X, Y}>\frac14\} \cup\{(X, Y)\ :\ X>1, Y\ge 0\} \\
            &\cup \{(X, Y)\ :\ X<1, Y\le 0\},
        \end{split}
    \end{equation}
    while $S(X, Y)<\infty$ in the region 
    \begin{equation}\label{eq:Dom_S}
        \operatorname*{Dom}(S)=\{(X, Y)\in\mathbb R^2\ : E_{X, Y}<\frac14 , \lvert X\rvert >1\}.
    \end{equation}
\end{remark}

We conclude this subsection with some considerations on Remark~\ref{rem:completely_explicit}; these will not be needed in the rest of the paper. The solutions to the Duffing equation~\eqref{eq:DuffingIVP} that satisfy $E_{X, Y}<1/4$ and $\lvert X\rvert <1$ are defined for all $s\in\mathbb R$ and can be expressed in terms of the Jacobi elliptic sine, denoted by $\operatorname*{sn}$:
\begin{equation}\label{eq:explicit_global}
    \begin{array}{ccc}
        U(s)=2 A \operatorname*{sn}(\omega s + \theta, k^2), & \text{where }\omega^2=1-\frac{A^2}{2},\,  k^2=\frac{A^2}{2-A^2},
    \end{array}
\end{equation}
for arbitrary $\lvert A\rvert <\sqrt 2$ and $\theta\in\mathbb R$. We recall that the function $\operatorname*{sn}$ is defined in terms of the inverse of an elliptic integral; \begin{equation}\label{eq:JacobiSN}
    \begin{array}{cc}
        \operatorname*{sn}(u, k^2)=\sin(\phi), & \displaystyle\text{where }u=\int_0^\phi\frac{d\theta}{\sqrt{1-k^2\sin(\theta)}}.
    \end{array}
\end{equation}
The fact that~\eqref{eq:explicit_global} indeed solves~\eqref{eq:DuffingIVP} is well-known and can be easily checked with a basic computer assisted computation.

\subsection{The threshold function $\beta$} \label{sec:subsection_beta}

To conclude the proof of Theorem~\ref{thm:main}, 
we need to express the level sets~\eqref{eq:level_sets_Tplus} in terms of a single threshold function $\beta=\beta(X)$.

We begin with the case $Y\ge 0$; with this condition, $\Phi\colon (X, Y)\mapsto (X, E_{X, Y})$ is bijective onto its range $\Ran(\Phi)$, with inverse 
\begin{equation}\label{eq:inverse_Phi}
    \begin{array}{cc}
        \Phi^{-1}(X, E)=(X, \sqrt{2E-X^2+\frac12 X^4}), &\forall (X, E)\in \Ran(\Phi).
    \end{array}    
\end{equation}
Recall that $\Ran(\Phi)=\{2E\ge X^2 -\tfrac12 X^4\}$, and note that the points at the boundary of $\Ran(\Phi)$ correspond to $Y=0$. With slight abuse of notation we will regard $T_\pm$, $R$ and $S$ as functions of $(X, E)$, implicitly assuming the change of variable~\eqref{eq:inverse_Phi}. Therefore the formulas of Lemma~\ref{lem:T_Plus_Expression} read in our $Y\ge 0$ case as 
\begin{equation}\label{eq:T_Plus_Expression_XE}
    T_+(X, E)=
    \begin{cases}
        R(X, E)=\int_{X}^\infty\frac{dv}{\sqrt{2E-v^2+\frac12 v^4}}, &E\ge \frac14,\\
        \text{as above}, &
        E<\frac14,X>0,\\
        S(X, E)=\left(\int_{\sqrt{1+\sqrt{1-4E}}}^{\lvert X\rvert} + \int_{\sqrt{1+\sqrt{1-4E}}}^\infty \right)\frac{dv}{\sqrt{2E-v^2+\frac12 v^4}},&E<\frac14, X<0;
    \end{cases}
\end{equation}
see Figure~\ref{fig:TPlusXE}. We also recall Corollary~\ref{cor:Tplus_Plus_Tminus}:
\begin{equation}\label{eq:T_Plus_plus_T_Minus}
    T_+(X, E)+\lvert T_-(X, E)\rvert=
    \begin{cases}\displaystyle
        2\int_0^\infty \frac{dv}{\sqrt{2E-v^2+\frac12 v^4}}, & E>\frac14, \\ \infty, & E=\frac14,\\\displaystyle
        2\int_{\sqrt{1+\sqrt{1-4E}}}^\infty \frac{dv}{\sqrt{2E-v^2+\frac12 v^4}}, & E<\frac14.
    \end{cases}
\end{equation}

It follows from~\eqref{eq:T_Plus_plus_T_Minus} that $T_+(X, E)+\lvert T_-(X, E)\rvert$ is decreasing in $E$ for $E>\frac14$, which is obvious, and it is increasing in $E$ for $E<\frac14$. Indeed
\begin{equation}\label{eq:TPlusTMinusDecreasing}
    \int_{\sqrt{1+\lambda}}^\infty \frac{dv}{\sqrt{2E-v^2+\frac12 v^4}}=\sqrt 2 \int_0^\infty \frac{dw}{\sqrt{w^2+2w\sqrt{1+\lambda}}\sqrt{w^2+2w\sqrt{1+\lambda}+2\lambda}}, 
\end{equation}
where $\lambda=\sqrt{1-4E}$, and the right-hand side of this expression is clearly decreasing in $\lambda$.
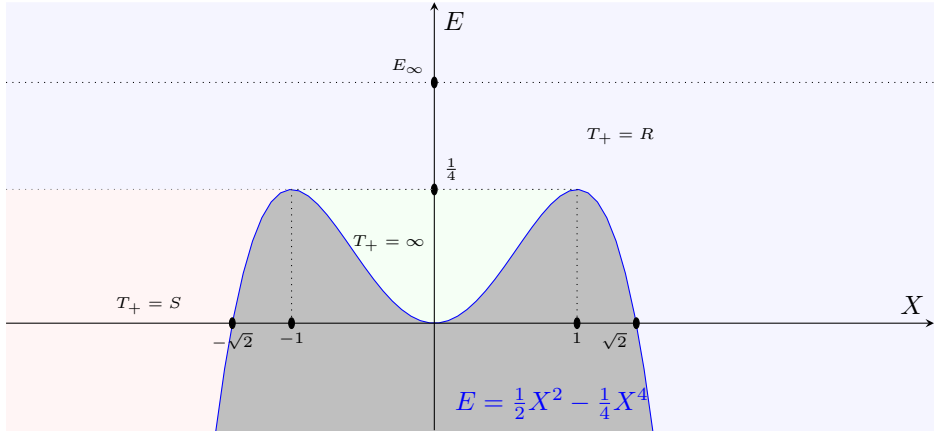
\begin{figure}
    \centering
    \begin{tikzpicture}
            %
            %
            %
            \begin{axis}[axis on top, axis x line=middle, axis y line=middle, 
            xmin=-3, xmax=3.5, ymin=-0.4, ymax=1.2, 
            xlabel=$\scriptsize{X}$, ylabel=$\scriptsize{E}$, 
            ytick=\empty, xtick=\empty, 
            xscale=1.8, 
            samples=50]
                %
                %
                %
                \addplot+[mark=none, domain=-1.6:1.6, draw=none]{-0.7168} \closedcycle;
                \addplot+[mark=none, 
                fill=lightgray, 
                draw=blue, domain=-1.6:1.6, name path=energy] {x^2-0.5*x^4};
                \draw[dotted] (-1,0)  node[anchor=north]{\tiny $-1$}--(-1, 0.5)--(1,0.5)--(1,0) node[anchor=north]{\tiny $1$};
                \draw[fill] (-1,0) circle(0.02);
                \draw[fill] (1,0) circle(0.02);
                \draw[fill] (0, 0.5) circle(0.02) node[anchor=south west]{\tiny $\tfrac14$};

                \draw[blue] (1.57,-0.3) node[left]{$E=\tfrac12 X^2-\tfrac14 X^4$};
                %
                %
                \draw[dotted] (-3, 0.9)--(3.5, 0.9);
                \draw[fill] (0, 0.9) circle(0.02) node[above left]{\tiny $E_\infty$};
                \draw[fill] (1.415, 0) circle(0.02) node[below left]{\tiny $\sqrt 2$};
                \draw[fill] (-1.415, 0) circle(0.02) node[below]{\tiny $-\sqrt 2$};
                \draw[dotted, gray] (-3, 0.5)--(1, 0.5);
                
                \addplot[draw=none, name path=left, domain=-3:1]{0.5};
                \addplot[draw=none, name path=bottom, domain=-3:3.5]{-0.7168};
                \addplot[red!10, opacity=0.4] fill between[of=left and energy, soft clip={domain=-1.6:-1}];
                \addplot[red!10, opacity=0.4] fill between[of=left and bottom, soft clip={domain=-3:-1.6}];

                \addplot[green!10, opacity=0.4] fill between[of=left and energy, soft clip={domain=-1:1}];
                %
                \addplot[draw=none, name path=up, domain=-3:3.5]{\pgfkeysvalueof{/pgfplots/ymax}};
                \addplot[blue!10, opacity=0.4] fill between[of=left and up, soft clip={domain=-3:1}];
                \addplot[blue!10, opacity=0.4] fill between[of=energy and up, soft clip={domain=1:1.6}];
                \addplot[blue!10, opacity=0.4] fill between[of=bottom and up, soft clip={domain=1.6:3.5}];
                %
                \draw (-2, 0) node[above]{\tiny $T_+=S$};
                \draw (0, 0.3) node[left]{\tiny $T_+=\infty$};
                \draw (1, 0.7) node[right]{\tiny $T_+=R$};
            \end{axis}
        \end{tikzpicture}
    \caption{The maximal positive time $T_+(X, E)$ equals $S$, $R$ or $\infty$ depending on the location of $(X, E)$. This picture depicts the $Y\ge 0$ case.}
    \label{fig:TPlusXE}
\end{figure}
We claim that this implies the existence of a $E_\infty>0$ such that, if $E\notin(0, E_\infty)$, then $T_+(X, E)<\pi$. Indeed, letting $E_\infty$ denote the unique solution to the equation 
\begin{equation}\label{eq:E_Infty}
    2\int_0^\infty \frac{dv}{\sqrt{2E_\infty - v^2+\frac12v^4}}=\pi,
\end{equation}
and noting that for $E=0$ we have 
\begin{equation}\label{eq:EnergyZeroTotalTime}
    \left.\big(T_++\lvert T_-\rvert\big)\right|_{E=0} = 2\sqrt 2\int_{\sqrt 2}^\infty \frac{dv}{v\sqrt{v^2-2}}=\pi,
\end{equation}
we conclude that 
\begin{equation}
    \begin{array}{ccc}
        T_+(X, E)<\pi, &\text{if }E\le0& \text{or }E\ge E_\infty.
    \end{array}
\end{equation}

We now turn to the strip $E\in (0, E_\infty)$. To begin, we observe that at the boundary of $\Ran(\Phi)$ there are exactly two points where $T_+=\pi$. To prove this we compute
\begin{equation}\label{eq:BoundaryTPlus}
    T_+(X, E)\big|_{E=\frac12X^2-\frac14X^4}=\sqrt{2}\int_0^\infty \frac{dw}{\sqrt{X^2(\cosh^2(w)+1)-2}},
\end{equation}
and we note that there are precisely two values $X=\pm X_C$ such that the right-hand integral equals $\pi$; indeed, it is an even function of $X$ that is strictly decreasing for $X\in(1, \sqrt{2})$, it tends to $+\infty$ at $X=1$ (corresponding to the constant solution $U_{1, 0}=1$) while it equals $\pi/2$ at $X=\sqrt 2$ (corresponding to the explicit solution $\sqrt{2}(\sin(\frac\pi 2- s))^{-1}$).

Still inside the strip $E\in (0, E_\infty)$, in the region where $T_+=R$ (recall Figure~\ref{fig:TPlusXE}), there is a strictly decreasing function $\tilde{\beta}_+$ such that 
\begin{equation}\label{eq:BetaTildePlus}
    \begin{array}{ccc}
        T_+(X, E)< \pi\, [\text{resp. }T_+(X, E)\ge \pi] & \iff & E> \tilde{\beta}_+(X)\, [\text{resp. }E\le\tilde{\beta}_+(X)],
    \end{array}
\end{equation}
because $R=R(X, E)$ is manifestly a strictly decreasing function of $X$ and $E$ separately. At $X=X_C$, we have $\tilde{\beta}_+(X_C)=\frac12X_C^2-\frac14X_C^4$ and  $\tilde{\beta}_+$ ceases to exist for $X>X_C$; indeed for $X>X_C$ we must have $R(X, E)<\pi$, by monotonicity. 

On the other hand, in the region where $T_+(X, E)=S(X, E)$, we claim that there is a strictly increasing function $\tilde{\beta}_-(X)$ such that
\begin{equation}\label{eq:BetaTildeMinus}
    \begin{array}{ccc}
        T_+(X, E)< \pi\, [\text{resp. }T_+(X, E)\ge\pi] & \iff & E< \tilde{\beta}_-(X)\, [\text{resp. }E\ge\tilde{\beta}_-(X)],
    \end{array}
\end{equation}
and $\tilde{\beta}_-$ ceases to exist at $X=-X_C$. To prove this claim we argue like in the previous case. The only difference is that $S(X, E)$ is strictly increasing in both $\lvert X\rvert $ and $E$; this can be seen by writing $\lambda=\sqrt{1-4E}$ and performing the change of variable $w=v-\sqrt{1+\lambda}$, yielding
\begin{equation}\label{eq:S_Change_Vars}
    \begin{split}
        S(X, Y)&= \int_{\sqrt{1+\lambda}}^{\lvert X\rvert}+\int_{\sqrt{1+\lambda}}^\infty \frac{dv}{\sqrt{2E -v^2+\frac12 v^4}} \\
        &= \left(\int_0^{\lvert X\rvert -\sqrt{1+\lambda}} + \int_0^\infty\right) 
        \frac{dw}{\sqrt{ (w^2+2w\sqrt{1+\lambda})(w^2+2w\sqrt{1+\lambda}+2\lambda)}},
    \end{split}
\end{equation}
which is manifestly a decreasing function of $\lambda$ and an increasing function of $\lvert X\rvert$. A plot of the graphs of $\tilde\beta_\pm$ is in  Figure~\ref{fig:ProofThmOne}.  Note that $\tilde\beta_-(X)\le \frac14<\tilde\beta_+(X)$ for all $X\le -X_C$.
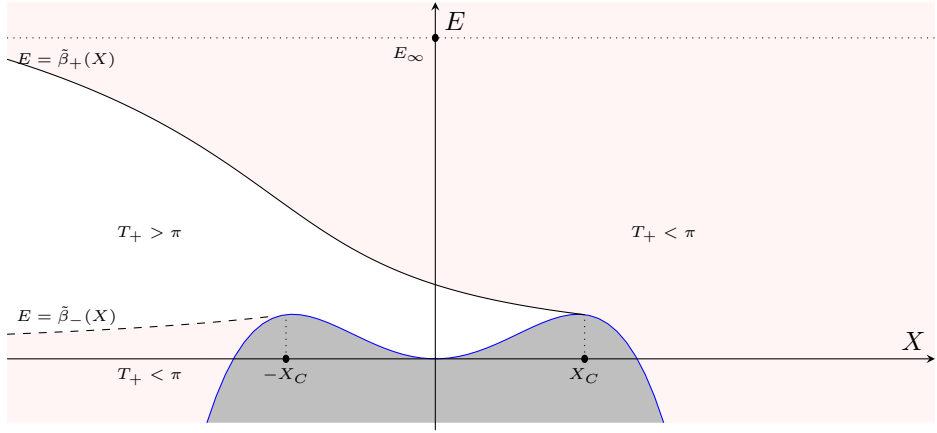
\begin{figure}
    \centering
          \begin{tikzpicture}
          \tikzmath{\EInfinity = 1.8; \UCritical = 1.046;} 
            \begin{axis}[axis on top, axis x line=middle, axis y line=middle, 
            xmin=-3, xmax=3.5, ymin=-0.4, ymax=2, 
            xlabel=$\scriptsize{X}$, ylabel=$\scriptsize{E}$, 
            ytick=\empty, xtick=\empty, 
            xscale=1.8,
            samples=50]
                %
                %
                %
                \addplot+[mark=none, domain=-1.6:1.6, draw=none]{-0.358} \closedcycle;
                \addplot+[mark=none, fill=lightgray, draw=blue, domain=-1.6:1.6, name path=energy] {0.5*x^2-0.25*x^4};
                %
                
                \draw[dotted] (-3, \EInfinity)--(3.5, \EInfinity);
                \draw[fill] (0, \EInfinity) circle(0.02) node[below left]{\tiny $E_\infty$};
                \draw (-\UCritical, 0) node[anchor=north]{\tiny $-X_C$};
                \draw (\UCritical, 0) node[anchor=north]{\tiny $X_C$};
                \draw[fill] (-\UCritical,0) circle(0.02);
                \draw[fill] (\UCritical,0) circle(0.02);
                \draw[dotted] (-\UCritical,0) -- (-\UCritical, 0.247);
                \draw[dotted] (\UCritical,0) -- (\UCritical, 0.247);
                %
                \addplot[dashed, domain=-3:-\UCritical, name path=bottom] table[col sep=comma]{ExportBOTTOM.txt};
                %
                
                 
                \addplot[domain=-3:\UCritical, name path=top] table[col sep=comma]{exportUP.txt}; 
                %
                %
                \addplot[draw=none, name path=left, domain=-3:-1.6]{-0.358};
                \addplot[draw=none, name path=right, domain=1.6:3.5]{-0.358};
                \addplot[draw=none, name path=up, domain=-3:3.5]{3};
                \addplot[red!10, opacity=0.4] fill between[of=left and bottom, soft clip={domain=-3:-1.6}];
                \addplot[red!10, opacity=0.4] fill between[of=energy and bottom, soft clip={domain=-1.6:-\UCritical}];
                \addplot[red!10, opacity=0.4] fill between[of=up and top, soft clip={domain=-3:\UCritical}];
                \addplot[red!10, opacity=0.4] fill between[of=up and energy, soft clip={domain=\UCritical:1.6}];
                \addplot[red!10, opacity=0.4] fill between[of=up and right, soft clip={domain=1.6:3.5}];
                %
                \draw (-3, 0.13) node[anchor=south west]{\tiny $E=\tilde{\beta}_-(X)$};
                \draw (-3, 1.68) node[anchor=west]{\tiny $E=\tilde{\beta}_+(X)$};
                \draw(-2, 0.7) node {\tiny $T_+>\pi$};
                \draw (1.6,0.7) node{\tiny $T_+<\pi$};
                \draw (-2,-0.1) node{\tiny $T_+<\pi$};
            \end{axis}
        \end{tikzpicture}
    \caption{The blow-up (lightly shaded) and scattering (white) regions of Theorem~\ref{thm:main} in the $(X, E)$ plane for $Y\ge 0$. The case $Y<0$ is obtained by the reflection $X\mapsto -X$.
    }
    \label{fig:ProofThmOne}
\end{figure}

We conclude the analysis of $\tilde\beta_\pm$ by noting that 
\begin{equation}\label{eq:betatilde_limit}
    \begin{array}{cc}
        \displaystyle \lim_{X\to -\infty} \tilde{\beta}_+(X)=E_\infty, & \displaystyle\lim_{X\to -\infty} \tilde{\beta}_-(X)=0.
    \end{array}
\end{equation}
Indeed, as $X\to -\infty$ we have $T_+(X, E)\to T_+(0, E) + \lvert T_-(0, E)\rvert$  (see~\eqref{eq:T_Plus_plus_T_Minus}), and we already saw that the latter equals $\pi$ for $E=E_\infty$ or $E=0$. 

The case $Y<0$ is obtained by the previous one by the symmetry $U_{-X, -Y}=-U_{X, Y}$; the level sets of $T_+$ are obtained from the previous ones via the reflection $X\mapsto -X$. So, when $Y<0$ and $E\ge \frac14$, or $Y<0$ and $E<\frac14$ and $X<0$,
\begin{equation}\label{eq:BetaTildePlus_negative}
    \begin{array}{ccc}
        T_+(X, E)< \pi\, [\text{resp. }T_+(X, E)\ge \pi] & \iff & E> \tilde{\beta}_+(-X)\, [\text{resp. }E\le\tilde{\beta}_+(-X)].
    \end{array}
\end{equation}
When $Y<0$ and $E<\frac14$ and $X>0$, 
\begin{equation}\label{eq:BetaTildeMinus_negative}
    \begin{array}{ccc}
        T_+(X, E)< \pi\, [\text{resp. }T_+(X, E)\ge \pi] & \iff & E< \tilde{\beta}_-(-X)\, [\text{resp. }E\ge\tilde{\beta}_+(-X)].
    \end{array}
\end{equation}

We pull back these results to the plane $(X, Y)$ via the map $\Phi(X, Y)=(X, E_{X, Y})$, obtaining that the threshold function $\beta$ is
\begin{equation}\label{eq:BetaThreshold}
    \beta(X)=\begin{cases}
        \sqrt{2\tilde{\beta}_+(X)-X^2+\frac12X^4},& X\le X_C,\\ 
        -\sqrt{2\tilde{\beta}_-(-X)-X^2+\frac12X^4}, & X>X_C.
    \end{cases}
\end{equation}
Note in particular that $\beta(X_C)=0$. Since $\tilde{\beta}_\pm(X)$ have a limit at $X\to -\infty$, it is clear that $\beta(X)\to \mp \infty$ as $X\to \pm \infty$. Moreover, $\beta(X)> -\beta(-X)$, as we prove by distinguishing cases: for $-X_C<X\le X_C$, this is obvious as $\beta(X)\ge 0$ while $\beta(-X)<0$; for $X\le -X_C$ this follows from $\tilde\beta_-(X)<\tilde\beta_+(X)$, and the remaining case $X>X_C$ follows from the previous one by the change of variable $X\mapsto -X$.

It remains to prove that $\beta$ is decreasing. This is easier to see by going back to considering $R$ and $S$ as functions on the $(X, Y)$ plane, yielding the following alternative characterization of the function $\beta$:
\begin{itemize}
    \item for $X\le X_C$, $Y=\beta(X)$ is the unique nonnegative solution to $R(X, Y)=\pi$; 
    \item for $X>X_C$, $Y=\beta(X)$ is the unique negative solution to $S(X, Y)=\pi$. 
\end{itemize}
This immediately shows that $\beta$ is decreasing, because $R(X, Y)$ is increasing in each variable separately for $Y\ge 0$, and similarly, $S(X, Y)$ is increasing in each variable separately for $Y<0$. 

\subsection{Complementary results}  
We collect here two more results which we will need in the rest of the paper, starting with the following asymptotics for blow-up solutions to the Duffing initial value problem~\eqref{eq:DuffingIVP}.
\begin{lemma}\label{lem:asympt_duffing}
Let $U$ be a solution to~\eqref{eq:DuffingIVP} with maximal time of existence $(T_-,T_+)$. Then, if $T_+<\infty$, $U$, $\dot{U}$ and $\ddot{U}$ have the same sign $\pm$ close to $T_+$ and:
\begin{align}
    \label{eq:T52} U(s)&=\pm \frac{\sqrt{2}}{T_+-s}+O(T_+-s),\quad s\nearrow T_+\\
    \label{eq:T53} \dot{U}(s)&=\pm \frac{\sqrt{2}}{(T_+-s)^2}+O(1),\quad s\nearrow T_+\\
    \label{eq:T54} \left|U^{(k)}(s)\right|&\lesssim \frac{1}{(T_+-s)^{1+k}},\quad s\nearrow T_+
\end{align}
\end{lemma}
\begin{proof}
Since $U$ blows up in finite time $T_+$, $|U(s)|+|\dot{U}(s)|$ is not bounded on $[0,T_+)$. As a consequence, $U$ cannot be bounded on $[0,T_+)$ (since $\ddot{U}+U=U^3$ would imply that $\dot{U}$ is also bounded). Thus (changing $U$ into $-U$ if necessary) there exists a sequence $(s_n)_n$, with $0\leq s_n<T_+$, $s_n\to T_+$, such that 
$$\lim_{n}U(s_n)=+\infty,\quad \forall s\in [0,s_n), \;U(s)\leq U(s_n).$$
This implies that $\dot{U}(s_n)\geq 0$. Furthermore, if $n$ is large, $U(s_n)>1$ and thus $\ddot{U}(s_n)=U^3(s_n)-U(s_n)>0$. By the equation \eqref{eq:DuffingIVP} and a simple bootstrap argument, choosing $n$ large, we deduce:
$$ \forall s> s_n,\quad U(s)>1, \;\dot{U}(s)>0,\; \ddot{U}(s)>0.$$
This shows the first point of the lemma. Since $U$ (and thus $\dot{U}$) is not bounded, this also implies
\begin{equation}
    \label{eq:T51}
    \lim_{s\to\infty}U(s)=\lim_{s\to\infty} \dot{U}(s)=\infty.
\end{equation}
We next consider the conserved energy of \eqref{eq:DuffingIVP}, 
$$E=E_{X,Y}=\frac{1}{2} \dot{U}^2(s)+\frac{1}{2}U^2(s)-\frac{1}{4}U^4=\frac 12X^2+\frac 12 Y^2-\frac 14 X^4.$$
For $s<T_+$ close to $T_+$ we have, since $\dot{U}$ is positive 
\begin{equation}
    \label{eq:T60}
    \dot{U}(s)=\sqrt{2E-U^2+\frac 12 U^4}.
\end{equation}
Integrating between $s$ and $T_+$, and changing variables, we obtain
\begin{equation}
    \label{eq:T61}
    \int_{U(s)}^\infty \frac{dv}{\sqrt{2E-v^2+\frac{1}{2}v^4}}=T_+-s.
\end{equation}
Since $\frac{1}{\sqrt{2E-v^2+\frac{1}{2}v^4} }=\frac{\sqrt{2}}{v^2}\left(1+O(v^{-2})\right),\quad v\to\infty,$ we obtain
$$\int_{U(s)}^{\infty}\left(\frac{\sqrt{2}}{v^2}+O\left(\frac{1}{v^4}\right)\right)dv=T_+-s,\quad s\nearrow T_+,$$
and thus $\frac{\sqrt{2}}{U(s)}+O\left(\frac{1}{U^3(s)}\right)=T_+-s$, $s\to T_+$. This yields 
$$ U(s)=\frac{\sqrt{2}}{T_+-s}+O\left(\frac{1}{(T_+-s)U^2(s)}\right),\quad s\to T_+,$$
which implies \eqref{eq:T52}. Combining with \eqref{eq:T60}, we obtain \eqref{eq:T53}.

One can prove  \eqref{eq:T54} by induction, using the equation $\ddot{U}+U=U^3$ together with \eqref{eq:T52} and \eqref{eq:T53}. We omit the details.
\end{proof}
We conclude the section with the following functional properties of the initial data of the cubic wave equation~\eqref{eq:NLW_recall}.
\begin{lemma}\label{lem:general_spaces_initial_data}
\label{lem:LpHs}
Let $(X,Y)\in \R^2$ and consider
$$u_0(x)=u_{X,Y}(0,x)=\frac{2X}{1+|x|^2},\quad u_1(x)=\partial_tu_{X,Y}(0,x)=\frac{4Y}{(1+|x|^2)^2}.$$
Then:
\begin{align}
\label{eq:u0Lp}
    \text{for }X\ne 0,\ u_0\in L^p(\R^3)&\iff p>3/2,\\
\label{eq:u1Lp}    
    \text{for }Y\ne 0,\ u_1\in L^p(\R^3)&\iff p>3/4,\\
\label{eq:u0Hs}
    \text{for }X\ne 0,\ u_0\in \dot{H}^s(\R^3)&\iff s>-1/2,\\
\label{eq:u1Hs}
    \text{for }Y\ne 0,\ u_1\in \dot{H}^s(\R^3)&\iff s>-3/2.\\
\end{align}
\end{lemma}
\begin{proof}
The assertions \eqref{eq:u0Lp} and \eqref{eq:u1Lp} follow directly from the definitions of $u_0$ and $u_1$. To prove \eqref{eq:u0Hs} and \eqref{eq:u1Hs}, we will compute the Fourier transform of $u_0$ and $u_1$, using the convention $\hat{f}(\xi)=\int e^{-ix\cdot \xi}f(x)dx$. We recall that when $f\in L^1(\R^3)$ is radial, we have, denoting $r=|x|$ and $\rho=|\xi|$
\begin{equation}
    \label{eq:T20}
    \hat{f}(\rho)=\frac{4\pi}{\rho}\int_0^{\infty}\sin(\rho r)f(r)rdr.
\end{equation}
As a consequence, by direct computation, the Fourier transform of $r^{-1}e^{-r}$ is $\frac{4\pi}{\rho^2+1}$ and, by the Fourier inversion formula, the Fourier transform of $\frac{1}{r^2+1}$ is $2\pi^2\frac{e^{-{\rho}}} {\rho}$. Recalling $\lVert f\rVert_{\dot{H}^s(\R^3)}^2\simeq \int_0^\infty \lvert\widehat{f}(\rho)\rvert^2\rho^{2s+2}\, d\rho$, the claim \eqref{eq:u0Hs} follows immediately.

Similarly, the Fourier transform of $e^{-r}$ is $\frac{4\pi}{(1+\rho^2)^2}$ and thus the Fourier transform of $\frac{1}{(1+r^2)^2}$ is $2\pi^2e^{-\rho}$. This yields \eqref{eq:u1Hs}.
\end{proof}

\section{Threshold solutions}\label{sec:threshold}
In this section, we prove Theorems \ref{thm:self0}, \ref{thm:self0'} and \ref{thm:self2} about the threshold solutions. In all the section, we consider a solution $u=u_{X,Y}$ with $Y=\beta(X)$ or $Y=-\beta(-X)$, so $u$ is asymptotically self-similar for future times $t\to+\infty$ (case (Threshold) in Theorem~\ref{thm:main}). We denote by $U=U_{X,Y}$ the corresponding solution of the Duffing equation \eqref{eq:DuffingIVP}, and we recall from the previous section that the maximal time of existence $T_+$ of $U$ is exactly $\pi$. 

\subsection{Pointwise estimates}
Throughout all of this section we will need the following preliminary estimates.

\begin{prop}[Self-similar behaviour at the threshold]\label{prop:self1}
 Let $(X,Y)\in \mathbb{R}^2$ with $Y=\beta(X)$ or $Y=-\beta(-X)$ and $u=u_{X,Y}$. Then there is a $C>0$ such that the following estimates hold for large $t>0$:
 
\noindent If $\lvert x\rvert\le t-1$,
  \begin{equation}\label{eq:self_similar_behavior}
 \left\lvert u(t, x) - \frac{\sqrt 2}{t}\right\rvert + (t-|x|) \left\lvert \partial_tu(t,x)+\frac{\sqrt{2}}{t^2}\right\rvert\le \frac{C}{t(t-\lvert x \rvert)^2}.
 \end{equation}
 If $|x|\ge t+1$
 \begin{gather}
 \label{eq:T72}
\left\lvert u(t,x)-\frac{2X}{|x|^2-t^2}-\frac{4 t Y}{(|x|^2-t^2)^2}\right\rvert \le  
\begin{cases}
    \frac{C}{(|x|-t)^3|x|}, & \text{ if }X\neq 0,\\
    \frac{C}{(|x|-t)^4|x|}+\frac{C t^2}{|x|^3(|x|-t)^3},& \text{ if }X=0,
\end{cases}  
\\
\label{eq:T72'}
\left\lvert \partial_tu(t,x)-\frac{4tX}{\left(|x|^2-t^2\right)^{2}} +\frac{4\left(|x|^2+3t^2\right)Y}{\left(|x|^2-t^2\right)^{3}}\right\rvert  \leq \frac{Ct}{(|x|-t)^{4}|x|^{2}}.
\end{gather}
If $-1\leq t-\lvert x\rvert\le 1$ 
\begin{equation}
    \label{eq:T71}
    \left\lvert u(t,x)\right\rvert +\left\lvert \partial_t u(t,x)\right\rvert \le \frac{C}{|t|}. 
\end{equation}
\end{prop}
\begin{remark}\label{rem:asympt_self_similar_past}
    If $u=u_{X, Y}$ is asymptotically self-similar for past times $t\to -\infty$ (as in the (Threshold) case of Remark~\ref{rem:neg_times}), then $u=-v$ for a solution $v$ that is asymptotically self-similar for future times, as in Proposition~\ref{prop:self1}. Using this it is straightforward to derive the result analogous to Proposition~\ref{prop:self1} for past times, and we omit further details.
\end{remark}
\begin{proof}[Proof of Proposition \ref{prop:self1}]
\noindent\textbf{Proof of \eqref{eq:self_similar_behavior}}.
\noindent In what follows, $t$ is large with $t\geq r+1$, and the symbols $O(\cdot)$ have to be understood in this region. We will use:
\begin{equation}
    \label{eq:T73}
    \arctan \ell=\frac{\pi}{2}-\frac{1}{\ell}+O\left(\frac{1}{\ell^3}\right),\quad \ell\to\infty.
\end{equation}
We recall that $u(t,r)=\Omega(t,r)U(s)$, where 
\begin{equation}
\label{eq:T200}
    s=\arctan(t+r)+\arctan(t-r)=\pi-\frac{2t}{t^2-r^2}+O\left(\frac{1}{(t-r)^3}\right),
\end{equation}
and $\Omega=2\left(1+(t-r)^2\right)^{-1/2}\left(1+(t+r)^2\right)^{-1/2}$. 
By direct computations, 
\begin{equation}
    \label{eq:T202}
    \Omega=2(t^2-r^2)^{-1}\left( 1+O\left((t-r)^{-2}\right)\right).
\end{equation}
Furthermore, by \eqref{eq:T200} and the asymptotic $U(s)=\sqrt{2}(\pi-s)^{-1}+O(\pi-s)$ of Lemma~\ref{lem:asympt_duffing},
\begin{equation}
\label{eq:U}
U(s)=U\left(\pi-\frac{2t}{t^2-r^2}+O\left(\frac{1}{(t-r)^3}\right)\right)=\frac{\sqrt{2}(t^2-r^2)}{2t}+O\left(\frac{1}{t-r}\right).
\end{equation}
Combining \eqref{eq:T202} and \eqref{eq:U}, we obtain the estimate on $u$ in \eqref{eq:self_similar_behavior}. 

The proof of the estimate on $\partial_tu$ is similar. We have
\begin{equation}
    \label{eq:T191}
    \partial_t u=\partial_t\Omega U(s)+\Omega \frac{\partial s}{\partial t}\dot{U}(s).
\end{equation}
Differentiating the definition of $s$, we obtain
\begin{equation}
    \label{eq:T201}
    \frac{d s}{dt}=\frac{2(t^2+r^2)}{\left(t^2-r^2\right)^2}+O\left(\frac{1}{(t-r)^4}\right).
\end{equation}
Furthermore, by \eqref{eq:T53},
\begin{equation}
\label{eq:T220}
\dot{U}(s)=\dot{U}\left(\pi -\frac{2t}{t^2-r^2}+O\left(\frac{1}{(t-r)^3}\right)\right)=\frac{\sqrt 2 (t^2-r^2)^2}{4 t^2}+O(1).
\end{equation}
Combining \eqref{eq:T220} with \eqref{eq:T202} and \eqref{eq:T201} we obtain
\begin{equation}
    \label{eq:T230}
    \Omega \frac{\partial s}{\partial t}\dot{U}(s)=\frac{\sqrt{2}(t^2+r^2)}{t^2(t^2-r^2)} \left(1+O\left((t-r)^{-2}\right)\right).
\end{equation}
Differentiating $\log \Omega$, we obtain 
\begin{equation}
\label{eq:T203}    
\frac{\partial\Omega}{\partial t}=-\frac{t}{2} \Omega^3(1+t^2-r^2),
\end{equation}
and thus, combining with \eqref{eq:T202}, 
$\frac{\partial \Omega}{\partial t}=-4t\left(t^2-r^2\right)^{-2}\left(1+O\left((t-r)^{-2}\right)\right).$
Hence, by \eqref{eq:U},
\begin{equation}
    \label{eq:T211}
    U(s)\partial_t\Omega=-2\sqrt{2}(t^2-r^2)^{-1}\left(1+O\left((t-r)^{-2}\right)\right).
\end{equation}
Combining \eqref{eq:T191}, \eqref{eq:T230} and \eqref{eq:T211} we deduce the estimate on $\partial_tu$ in \eqref{eq:self_similar_behavior}.

\medskip

\noindent\textbf{Proofs of \eqref{eq:T72} and \eqref{eq:T72'}}.

We next assume $r\geq 1+t\gg 1.$
In this region, we have
\begin{equation}
\label{eq:T90}
    \Omega=2\left(r^2-t^2\right)^{-1}\left(1+O\left((r-t)^{-2}\right)\right).
\end{equation}
Furthermore,
by \eqref{eq:T73},
\begin{equation}
\label{eq:sbis}
    s=\frac{2t}{r^2-t^2}+O\left(\frac{1}{(r-t)^3}\right).
\end{equation}
As a consequence, recalling that $X=U(0)$ and $Y=\dot{U}(0)$ and using $U(s)=X+sY+O(s^2)$ for $|s|\leq 1$, we obtain
\begin{equation}
\label{eq:Ubis}
    U\left(s\right)=X+\frac{2t}{r^2-t^2}Y+O\left(\frac{t^2}{(r^2-t^2)^2}\right).
\end{equation}
Since $u(t,r)=\Omega U(s)$, we  deduce \eqref{eq:T72} from \eqref{eq:T90} and \eqref{eq:Ubis}

Also, still assuming $r\geq 1+t\gg 1$, we have, using \eqref{eq:T203},
\begin{equation}
    \label{eq:T242} \partial_t\Omega=\frac{4 t}{(r^2-t^2)^{2}}\left(1+O\left(\frac{1}{(r-t)^{2}}\right)\right),\quad \partial_t s=\frac{2(r^2+t^2)}{\left(r^2-t^2\right)^2}+O\left(\frac{1}{(r-t)^4}\right).
\end{equation}
Furthermore, by \eqref{eq:sbis}, we have
\begin{equation}
\label{U'bis}
\dot{U}(s)=Y+O(s)=Y+O\left(\frac{t}{(r-t)r}\right).
\end{equation}
By \eqref{eq:Ubis} and \eqref{eq:T242}, we have
\begin{equation*}
\partial_t \Omega U(s)=4t(r^2-t^2)^{-2}    X+8t^2(r^2-t^2)^{-3}Y+O\left(t(r-t)^{-4}r^{-2}\right).
\end{equation*}
By \eqref{eq:T90}, \eqref{eq:T242} and \eqref{U'bis}, 
$$ \Omega \frac{\partial s}{\partial t} \dot{U}(s)=4(r^2+t^2)\left(r^2-t^2\right)^{-3}Y +O\left((r-t)^{-5}r^{-1}\right)+O\left((r-t)^{-4}tr^{-2}\right).$$
The two last estimates yield \eqref{eq:T72'}.

\medskip

\noindent\textbf{Proof of \eqref{eq:T71}.} 

Since for $r-1\leq t\leq r+1$, $t\geq 0$, we have $0\leq s\leq \frac{\pi}{2}+\arctan 1=\frac{3\pi}{4}<\pi$, we see that $U(s)$ and $\dot{U}(s)$ remain bounded in this region. Also, we have
$$\Omega=O\left(\frac{1}{t\left\langle t-r\right\rangle^{\frac 12}}\right),\; \partial_t\Omega=O\left(\frac{1}{t\langle t-r\rangle^{3}}\right),\; \frac{\partial s}{\partial t}= O\left(1+\frac{1}{\langle t-r\rangle^2}\right),$$
where by definition $\langle y\rangle:=\sqrt{1+\lvert y\rvert^2}$, 
Recalling that $u=\Omega U(s)$ and the formula \eqref{eq:T191} for $\partial_tu$, we obtain the estimate \eqref{eq:T71}.
\end{proof}

\subsection{Asymptotics of Lebesgue norms}\label{sec:threshold_lebesgue}
Here we prove the claims~\eqref{eq:asymptoticLp} and~\eqref{eq:boundLp} on the $L^p$ norms of $u$ in Theorem~\ref{thm:self2}. We assume $p>3/2$, and $t\gg 1$. By \eqref{eq:T71},
\begin{equation}
    \label{eq:T110}
    \int_{t-1}^{t+1} |u(t,r)|^pr^2dr\lesssim t^{2-p}.
\end{equation}
Next, we see by \eqref{eq:T72} and the change of variable $\sigma=r/t$, that
$$\int_{t+1}^{\infty}|u(t,r)|^pr^2dr\lesssim \int_{t+1}^{\infty}\frac{1}{(r^2-t^2)^p}r^2dr\lesssim t^{3-2p} \int_{1+\frac{1}{t}}^\infty\frac{\sigma^2}{(\sigma^2-1)^p}d\sigma.$$
Since (using that $p>3/2$), 
$\int_{1+\frac{1}{t}}^\infty\frac{\sigma^2}{(\sigma^2-1)^p}d\sigma=O(t^{p-1}),$
we deduce
\begin{equation*}
    \int_{t+1}^{\infty} |u(t,r)|^pr^2dr=O(t^{2-p}).
\end{equation*}
Combining with \eqref{eq:T110}, we obtain \eqref{eq:boundLp}. To prove \eqref{eq:asymptoticLp}, we use \eqref{eq:self_similar_behavior}, which yields
\begin{multline*}
    \int_0^{t-1}|u(t,r)|^pr^2dr=\int_0^{t-1} \left|\frac{2}{t}+O\left(\frac{1}{t(t-r)^2}\right)\right|^pr^2dr\\
    =2^{\frac{p}{2}} \int_0^{t-1} \frac{1}{t^p}r^2dr+O\left(\int_0^{t-1}\frac{r^2}{t^p(t-r)^2}dr\right).
\end{multline*}
Computing the first integral explicitly and noting that $\int_0^{t-1}\frac{r^2}{(t-r)^2}dr=O(t^2)$ as $t\to\infty$, we obtain
\begin{equation}
    \label{eq:T130}
    \int_0^{t-1} |u(t,r)|^pr^2dr=\frac{2^{\frac{p}{2}}}{3}t^{3-p}+O(t^{2-p}), \quad t\to\infty,
\end{equation}
which, combined with \eqref{eq:boundLp} yields \eqref{eq:asymptoticLp}.

\subsection{Pointwise bounds on the derivatives}

\label{sub:rough}

We now give rough bounds of the derivatives of $u$, some of which will be needed to prove the remaining statements~\eqref{eq:asymptoticdtLp}, \eqref{eq:asymptoticHs}, \eqref{eq:asymptoticH1/2} and~\eqref{eq:asymptoticHs'}, and so finish the proof of Theorem \ref{thm:self2}.
\begin{prop}
\label{P:rough}
For $\ell_1+\ell_2\geq 1$ 
\begin{equation}
\label{superbound4'}
\left| \partial_t^{\ell_1}\partial_r^{\ell_2}u(t,r)\right|\lesssim \frac{1}{\langle r+t\rangle \langle r-t\rangle^{1+\ell_1+\ell_2}}.
\end{equation}
and, if furthermore $p\geq 1$, $(p,\ell_1+\ell_2)\neq (1,1)$,
\begin{equation}
    \label{boundtlLp}
\left\|\partial_t^{\ell_1}\partial_r^{\ell_2}u(t,r)\right\|_{L^p}\lesssim t^{\frac{2}{p}-1}.
\end{equation}
\end{prop}
\begin{proof}[Proof of Proposition~\ref{P:rough}]

\noindent \textbf{Proof of \eqref{superbound4'}.} To simplify the exposition, we focus on the derivatives with respect to $r$. The bounds are exactly the sames for the derivatives with respect to $t$, or combination of the two types of derivatives.
Since $u=\Omega U(s)$ of $u$, we obtain, for $\ell\geq 1$
\begin{equation}
    \label{productrule}
 \partial_r^{\ell}u=\sum_{k=0}^{\ell} \binom{\ell}{k} \partial_r^{\ell-k}\Omega \,\partial_r^{k}U(s).
\end{equation}
The $k$th derivative of $U$, $\partial_r^k U(s)$ is a linear combination of terms of the form $\prod_{i=1}^m \partial_r^{k_i}s U^{(m)}(s)$, where $1\leq m\leq k$, $k_i\geq 1$, $\sum_{i=1}^mk_i=k$.
For $j\geq 1$, we have
$$\partial^j_r s=\partial_r^{j-1} \left(\frac{1}{1+(r+t)^2}\right)-\partial_r^{j-1} \left(\frac{1}{1+(r-t)^2}\right),$$
and thus
\begin{equation}
    \label{superbound1}
    \left|\partial_r^js\right|\lesssim \frac{1}{(1+|r-t|)^{j+1}}.
\end{equation}
On the other hand, it is easy to check that
\begin{equation*} 
\begin{array}{cc} 
    \displaystyle s=\arctan (t+r)+\arctan(t-r)\leq \pi-\frac{1}{2(t-r)}, & \text{for }t-r\geq 1.
\end{array}
\end{equation*}
Using that all the derivatives of $U$ are bounded on $[0,\pi-1]$ and Lemma \ref{lem:asympt_duffing}, we deduce, for $m\geq 1$, 
$$ \left|U^{(m)}(s)\right|\lesssim (t-r)^m\indic_{\{t\geq r\}}+1.$$
Combining with \eqref{superbound1}, we deduce
$$
\left|\partial_r^k\left(U(s)\right)\right|\lesssim \sum_{m=1}^k \frac{1}{(1+|r-t|)^{m+k}}\left(|t-r|^m+1\right)\lesssim \frac{1}{\langle r-t\rangle^k}.  
$$
Furthermore, using the definition of $\Omega$, we see that 
$\left|\partial_r^j\Omega\right|\lesssim \frac{1}{\langle r+t\rangle \langle t-r\rangle^{j+1}}$.
Going back to \eqref{productrule}, we obtain \eqref{superbound4'} when $\ell_1=0$, $\ell_2\geq 1$. The same proof yields \eqref{superbound4'} when $\ell_1+\ell_2\geq 1$.

\noindent\textbf{Proof of \eqref{boundtlLp}}. We let $p\geq 1$, and let $(\ell_1,\ell_2)$ be two integers with $\ell=\ell_1+\ell_2\geq 1$. Then by \eqref{superbound4'},
\begin{equation*}
    \int_0^{\infty} \left|\partial_t^{\ell_1}\partial_r^{\ell_2}u(t,r)\right|^pr^2dr\lesssim \int_0^{\infty}\frac{1}{\langle r+t\rangle^p\langle r-t\rangle^{(1+\ell)p}} r^2dr.
\end{equation*}
Assuming $(p,\ell)\neq (1,1)$, we have $p>\frac{3}{2+\ell}$, so that the preceding integral is finite. The estimate \eqref{boundtlLp} then follows from the fact that for large $t$,
\begin{equation}
    \label{superbound5}
    \int_0^{\infty}\frac{1}{\langle r+t\rangle^p\langle r-t\rangle^{(1+\ell)p}} r^2dr\lesssim t^{2-p},
\end{equation}
which can be proved by estimating separately the integrals $\int_0^{t-1}$, $\int_{t-1}^{t+1}$ and $\int_{t+1}^{\infty}$.
\end{proof}

The upper bound in \eqref{eq:asymptoticdtLp} is a consequence of Proposition \ref{P:rough} if $p>1$. So we turn to the upper bound in \eqref{eq:asymptoticdtLp} when $p=1$, that is
\begin{equation}
\label{eq:upperdtLp}
\|\partial_tu(t)\|_{L^1}=O(t).    
\end{equation}
From \eqref{superbound4'} we obtain immediately
\begin{equation}
    \label{eq:dtLp1}
    \int_{0}^{t+1}|\partial_tu(t,r)|r^2dr\lesssim t
\end{equation}
Also, noting that if $b>\max(a+1,1)$, we have
$\int_{t+1}^{\infty}\frac{r^a}{(r-t)^b}dr\lesssim t^a$, we can prove, bounding separately the terms 
$$\int_{t+1}^\infty \frac{t}{\left(|x|^2-t^2\right)^{2}}r^2dr,\quad \int_{t+1}^{\infty} \frac{r^{2}}{\left(|x|^2-t^2\right)^{3}} r^2dr,\quad \int_{t+1}^\infty \frac{t}{(r-t)^{4}r^{2}}r^2dr.
$$
arising from \eqref{eq:T72'}, we obtain
\begin{equation}
    \label{eq:dtLp2}
    \int_{t+1}^{\infty}|\partial_tu(t,r)|r^2dr\lesssim t,
\end{equation}
which yields \eqref{eq:upperdtLp}. The lower bound in \eqref{eq:asymptoticdtLp} will be proved in Subsection \ref{sub:radiation}.

\subsection{Asymptotics of Sobolev norms}\label{sec:threshold_sobolev}
We prove here the parts of the statements \eqref{eq:asymptoticHs}, \eqref{eq:asymptoticH1/2} and \eqref{eq:asymptoticHs'} in Theorem \ref{thm:self2} that concern the Sobolev norms of $u$. The remaining parts, concerning the Sobolev norms of $\partial_t u$, will be proved in the next subsection. 

By conservation of the energy, 
$$\frac12\|\boldsymbol{u}(t)\|_{\mathcal H^1}^2=\frac{1}{2}\|\nabla u(t)\|^2_{L^2}+\frac{1}{2}\left\|\partial_tu\right\|_{L^2}^2=E(u_0,u_1)+\frac 14\|u(t)\|_{L^4}^4.$$
Combining with \eqref{eq:asymptoticLp} with $p=4$, we obtain
\begin{equation}
    \label{eq:T142}
    \|\boldsymbol{u}(t)\|^2_{\mathcal{H}^1} =2E(u_0,u_1)+O(t^{-1}),\quad t\to\infty.
\end{equation}

Let $\nu\in [0,1]$. We have 
$$\|u(t)\|^2_{\dot{H}^{\nu}}=4\pi \int_0^{\infty}|\hat{u}(t,\rho)|^2\rho^{2+2\nu}d\rho, $$
and moreover
$$\int_1^{\infty} |\hat{u}(t,\rho)|^2\rho^{2+2\nu}d\rho\leq \int_1^{\infty}|\hat{u}(t,\rho)|^2\rho^4d\rho,$$
thus by \eqref{eq:T142},
\begin{equation}
    \label{eq:T150}
    \limsup_{t\to\infty} \int_1^{\infty}|\hat{u}(t,\rho)|^2\rho^{2+2\nu}d\rho<\infty.
\end{equation}
We decompose 
$$u=\underbrace{u \indic_{r\leq t-1}}_{u_{<}}+\underbrace{u\indic_{r>t-1}}_{u_>}.$$
By \eqref{eq:boundLp}, $\|u_{>}(t)\|_{L^2}=O(1)$. As a consequence, as $t\to\infty$
\begin{equation}
    \label{eq:T151}
    \int_0^{1}|\widehat{u_>}(t,\rho)|^2\rho^{2+2\nu}d\rho \leq \int_0^1\left|\widehat{u_>}(t,\rho)\right|^2\rho^2d\rho=O(1).
\end{equation}
On the other hand, by the formula \eqref{eq:T20} for the radial Fourier transform, then \eqref{eq:self_similar_behavior},
\begin{equation}
    \widehat{u_<}(t,\rho)=\frac{4\pi}{\rho}\int_0^{t-1}\sin(\rho r)u(t,r)rdr=\frac{4\pi}{\rho}\int_0^{t-1}\sin(\rho r)\frac{\sqrt{2} r}{t}dr+O\left(\frac{1}{\rho}
\int_0^{t-1}\frac{r}{t(t-r)^2}dr\right).
\end{equation}
By the change of variable $t=t\sigma$, $\int_0^{t-1}\frac{r}{t(t-r)^2}dr=O(1)$ as $t\to\infty$. On the other hand, by direct computations
$$\int_0^{t-1}\sin(\rho r) rdr=\frac{1}{\rho^2}\sin(\rho(t-1))-\frac{t-1}{\rho}\cos(\rho(t-1)).$$
Thus 
\begin{equation}
    \label{eq:T160}
    \widehat{u_<}(t,\rho)=\frac{4\sqrt{2}\pi}{\rho^3t}\left(\sin(\rho(t-1))-\rho(t-1)\cos(\rho(t-1))\right)+\frac{1}{\rho}O(1),
\end{equation}
where $O(1)$ is uniform for $t\gg 1$ and $\rho\in [0,1]$. 

We also have, by the change of variable $\sigma=\rho(t-1)$
\begin{multline}
    \label{eq:T161}
    \int_0^1\frac{1}{\rho^6}\left( \sin(\rho(t-1))-\rho(t-1) \cos(\rho(t-1))\right)^2 \rho^{2+2\nu}d\rho\\
    =\frac{1}{(t-1)^{2\nu-3}}\int_0^{t-1}(\sin \sigma -\sigma\cos \sigma)^2\sigma^{-4+2\nu}d\sigma.
\end{multline}

\medskip

\noindent\emph{Case 1: $\frac{1}{2}<\nu \leq 1$}
$$\int_0^{t-1} (\sin \sigma-\sigma\cos\sigma)^2\sigma^{-4+2\nu}d\sigma\leq \int_0^{t-1}(1+\sigma)^2\sigma^{-4+2\nu}d\sigma\lesssim t^{2\nu-1}.$$
Thus 
$$\int_0^{1}\frac{1}{\rho^6}\left(\sin(\rho(t-1))-\rho(t-1)\cos(\rho(t-1))\right)^2\rho^{2+2\nu}d\rho\lesssim t^2.$$
In view of \eqref{eq:T160}, we deduce
\begin{equation}
    \label{eq:T170}
    \int_0^1 \left|\widehat{u_{<}}(t,\rho)\right|^2\rho^{2+2\nu}d\rho=O(1),\quad t\to\infty.
\end{equation}
By \eqref{eq:T150}, \eqref{eq:T151} and \eqref{eq:T170},
\begin{equation}
    \label{eq:T171}
    \frac 12<\nu\leq 1\Longrightarrow \|u(t)\|_{\dot{H}^{\nu}}=O(1),\quad t\to\infty.
\end{equation}
Combining with \eqref{boundtlLp}, we see that for any $\nu>1/2$, $\|u(t)\|_{\dot{H}^{\nu}}=O(1)$. We will prove below (see Remark \ref{rem:nonradiative}) that $1\lesssim \|u(t)\|_{\dot{H}^{1}}$, which, by a simple interpolation argument, will imply the first estimate in \eqref{eq:asymptoticHs'}.

\medskip

\noindent\emph{Case 2: $0\leq \nu< \frac 12$}

In this case, we have 
$$ \int_0^{t-1}\left(\sin \sigma -\sigma \cos \sigma\right)^2\sigma^{-4+2\nu}d\sigma=c_{\nu}^{2}-\int_{t-1}^{\infty} \left(\sin \sigma-\sigma \cos \sigma\right)^2\sigma^{-4+2\nu}d\sigma=c_{\nu}^2+O\left(t^{2\nu-1}\right),$$ 
where $c_{\nu}^2=\int_0^{\infty} (\sin \sigma-\sigma \cos\sigma)^2\sigma^{-4+2\nu}d\sigma$. By \eqref{eq:T160} and \eqref{eq:T161},
$$\int_0^1 \left|\widehat{u_{<}}(t,\rho)\right|^2\rho^{2+2\nu}d\rho=\frac{32\pi^2}{t^{2\nu-1}}c_{\nu}^2+O\left(\frac{1}{t^{\nu-\frac 12}}\right).$$
Combining with \eqref{eq:T150} and \eqref{eq:T151}, we deduce the first estimate in \eqref{eq:asymptoticHs}.

\medskip

\noindent\emph{Case 3: $\nu=1/2$}

We have 
$$\int_0^{t-1}(\sin \sigma -\sigma \cos\sigma)^2\sigma^{-3}d\sigma=\int_0^1(\sin\sigma-\sigma\cos\sigma)^2\sigma^{-3}d\sigma+\int_1^{t-1} (\sin \sigma -\sigma \cos\sigma)^2\sigma^{-3}d\sigma.$$
The first integral is fixed, and finite. The second integral satisfies
\begin{multline*}
    \int_1^{t-1} (\sin \sigma -\sigma \cos\sigma)^2\sigma^{-3}d\sigma=\int_1^{t-1}\frac{\cos^2\sigma}{\sigma}d\sigma+O(1)\\
    =\frac 12\int_{1}^{t-1}\frac{d\sigma}{\sigma}+\frac{1}{2}\int_{1}^{t-1}\frac{\cos(2\sigma)}{\sigma}d\sigma=\frac{1}{2}\log (t-1)+O(1)=\frac{1}{2}\log t+O(1).
\end{multline*}
In view of \eqref{eq:T160}, \eqref{eq:T161} we obtain
$$\int_0^1\left|\widehat{u_{<}}(t,\rho)\right|^2\rho^3d\rho =16\pi^2\log t+O\left(\sqrt{\log t}\right),$$
which gives the first estimate in \eqref{eq:asymptoticH1/2}.

\subsection{Asymptotics of Sobolev norms of the time derivative}\label{sec:threshold_sobolev_timeder}
We next prove the estimates on the Sobolev norms of $\partial_tu$ in \eqref{eq:asymptoticHs}, \eqref{eq:asymptoticH1/2} and \eqref{eq:asymptoticHs'}. 

Let $\nu\in [-1,0]$. By \eqref{eq:T142}, $\|\partial_t\hat{u}(t)\|_{L^2}=(2\pi)^\frac32\|\partial_t u (t)\|_{L^2}=O(1)$. Thus 
\begin{equation}
    \label{dt500}
    \limsup_{t\to\infty} \int_{1}^{\infty} |\partial_t\hat{u}(t,\rho)|^2\rho^{2\nu+2}d\rho<\infty.
\end{equation}
Next, we recall from \eqref{superbound4'} that $|\partial_tu(t,r)|\lesssim \langle r+t\rangle^{-1}\langle r-t\rangle^{-2}$. By the formula \eqref{eq:T20} for the radial Fourier transform.
\begin{equation}
\label{dt10}
    \left|\partial_t\hat{u}(t,\rho)\right|\lesssim \frac{1}{\rho}\int \left|\sin(\rho r)\right|\frac{r}{\langle r+t\rangle \langle r-t\rangle^2}dr.
\end{equation}
Bounding $\sin(r\rho)$ by $r\rho$ if $r\rho\leq 1$ and by $1$ if $r\rho\geq 1$, we obtain 
$$ \int_{0}^1|\partial_t\hat{u}(\rho)|^2\rho^{2+2\nu}d\rho\lesssim A+B,$$
where 
\begin{align*}
A&=\int_0^1 \rho^{2+2\nu} \left(\int_0^{1/\rho}\frac{r^2}{\langle r+t\rangle \langle r-t\rangle^2}dr\right)^2d\rho\\ 
B&=\int_0^1 \rho^{2\nu} \left(\int_{1/\rho}^{\infty}\frac{r}{\langle r+t\rangle \langle r-t\rangle^2}dr\right)^2d\rho.    
\end{align*}
We decompose $A\lesssim A_1+A_2+A_3$, $B\lesssim B_1+B_2+B_3$, where the terms $A_j$ and $B_j$ are defined in the same way as $A$ and $B$, with the additional condition that $r\leq \frac{t}{2}$ (for $A_1$ and $B_1$), $\frac t2\leq r\leq 2t$ (for $A_2$ and $B_2$) and $r\geq 2t$ (for $A_3$ and $B_3$) in the interior integral. 

When $r\leq t/2$, we have $\langle r-t\rangle\approx t\approx \langle r+t\rangle$ and thus \begin{gather*}
    A_1\lesssim \int_0^1 \rho^{2+2\nu}\left(\int_0^{t/2}\frac{dr}{t}\right)^{2}d\rho\lesssim 1\\
    B_1\lesssim \int_{2/t}^1 \rho^{2\nu}\left(\int_{1/\rho}^{t/2} \frac{dr}{t^2}\right)^2d\rho \lesssim 1,
\end{gather*}
where we have used to bound $B_1$ that in the interior integral, one must have $1/\rho\leq r \leq t/2$, which imposes $\rho\geq 2/t$ (we will use the same argument to restrict the domains of integration in the bounds of $A_2$, $B_2$ and $A_3$ below).
When $\frac{t}{2}\leq r\leq 2t$, we have $r\approx \langle r+t\rangle \approx t$. Thus
\begin{equation*}
A_2\lesssim \int_0^{2/t}\rho^{2+2\nu}t^2\left(\int_{t/2}^{2t} \frac{dr}{\langle r-t\rangle^2}\right)^2d\rho \lesssim t^2\frac{1}{t^{3+2\nu}}\lesssim t^{-1-2\nu},
\end{equation*}
where we have used  $\int_{t/2}^{2t} \frac{dr}{\langle r-t\rangle^2}=\int_{-t/2}^{t} \frac{dr}{\langle r\rangle^2}\lesssim \int_{-\infty}^{\infty}\frac{dr}{\langle r\rangle^2}\lesssim 1$. Moreover
\begin{equation*}
B_2\lesssim \int_{1/2t}^1 \rho^{2\nu}\left(\int_t^{2t}\frac{dr}{\langle r-t\rangle^2}\right)^2d\rho \lesssim\begin{cases} 1&\text{ if }\nu >-1/2\\ \log t&\text{ if }\nu=-1/2\\ t^{-1-2\nu}&\text{ if } \nu<-1/2.\end{cases}.
\end{equation*}
When $r\geq 2t$, we have $\langle r-t\rangle \approx \langle r+t\rangle \approx r$. Thus 
\begin{gather*}
A_3\lesssim \int_0^{\frac{1}{2t}} \rho^{2+2\nu} \left(\int_0^{1/\rho}\frac{dr}{\langle r\rangle}\right)^2d\rho\lesssim t^{-3-2\nu}\log^2t\\
B_3\lesssim \int_0^1\rho^{2\nu}\left(\int_{1/\rho}^{\infty}\frac{dr}{r^2}\right)^2 d\rho \lesssim \int_0^1 \rho^{2\nu+2}d\rho\lesssim 1.
\end{gather*}
Combining the above bounds, we deduce 
$$\int_0^1 |\partial_t\hat{u}(\rho)|^2\rho^{2\nu+2}d\rho\lesssim \begin{cases} 1&\text{ if }\nu >-1/2\\ \log t&\text{ if }\nu=-1/2\\ t^{-1-2\nu}&\text{ if } \nu<-1/2,\end{cases}.$$
Combining with \eqref{dt500}, we deduce the bounds of the Sobolev norms of $\partial_tu$ in \eqref{eq:asymptoticHs} and \eqref{eq:asymptoticH1/2}, and also  
\begin{equation}
\label{bounddtu}
\|\partial_tu\|_{\dot{H}^{\nu}}\lesssim 1 
\end{equation}
for $\nu\in (-1/2,0]$. Recalling that \eqref{bounddtu} also holds when $\nu>1$ (see \eqref{boundtlLp}), and using the fact (shown below, see Remark \ref{rem:nonradiative}) that $1\lesssim \|\partial_tu(t)\|_{L^2}$, we deduce the estimate of the norms of $\partial_tu$ in \eqref{eq:asymptoticHs'}.

 This concludes the proof of Theorem \ref{thm:self2}, except for the lower bound in \eqref{eq:asymptoticdtLp} which will be proved at the end of the next subsection.

\subsection{Radiation term outside the wave cone}
\label{sub:radiation}
In this subsection, we prove Theorem \ref{thm:self0'}. We start with the following lemma:
\begin{lemma}
\label{lem:radiation}
Let $A\in \mathbb{R}$, $X\in \mathbb{R}$, $Y=\beta(X)$ or $Y=-\beta(-X)$ and $u=u_{X,Y}$. Let $U$ be the corresponding solution of the Duffing equation \eqref{eq:DuffingIVP}. Let 
$$ g(\eta)=-\frac{1}{(1+\eta^2)^{1/2}} U\left(\frac{\pi}{2}-\arctan\eta\right).$$
Then, for $p>1$,
\begin{multline}
\label{lim_radiation}
  \lim_{t\to\infty} \int_{A+t}^{\infty} \left(r\partial_t u(t,r)-g'(r-t)\right)^pdr\\
  =\lim_{t\to\infty} \int_{A+t}^{\infty} \left(r\partial_r u(t,r)+g'(r-t)\right)^pdr=0
 \end{multline}
\end{lemma}
\begin{proof}
We have
 \begin{equation}
  \label{exp_dtu_dru}
  \frac{\partial u}{\partial t}=\frac{\partial \Omega}{\partial t}U(s)+\Omega \frac{\partial s}{\partial t}\dot{U}(s),\quad \frac{\partial u}{\partial r}=\frac{\partial \Omega}{\partial r}U(s)+\Omega \frac{\partial s}{\partial r}\dot{U}(s).
 \end{equation}
 We fix a $A\in \mathbb{R}$ and assume without loss of generality $A<0$. 
 We let $r=t+\eta$ with $\eta\geq A$, and $t\geq -A$, so that $r+t=2t+\eta\geq \max(t,|\eta|)$. We have \footnote{Denoting as usual $\frac{\partial}{\partial t}$ the derivative with respect to $t$ at fixed $r$.} 
 \begin{gather}
 \label{estimrad1}
s=-\arctan \eta+\arctan(2t+\eta),\quad |s|\leq \frac{\pi}{2}+\arctan |A|<\pi\\
\label{estimrad2}
\frac{\partial s}{\partial t}=\frac{1}{1+\eta^2}+\frac{1}{1+(2t+\eta)^2},\quad \left|\frac{\partial s}{\partial t}\right|\leq \frac{2}{1+\eta^2}\\
\label{estimrad3}
\frac{\partial s}{\partial r}=-\frac{1}{1+\eta^2}+\frac{1}{1+(2t+\eta)^2},\quad \left|\frac{\partial s}{\partial r}\right|\leq \frac{1}{1+\eta^2}\\
\label{estimrad4}
(t+\eta)\Omega=\frac{2(t+\eta)}{\sqrt{1+(2t+\eta)^2}\sqrt{1+\eta^2}},\quad |(t+\eta)\Omega|\leq \frac{2}{\sqrt{1+\eta^2}}\\
\label{estimrad5}
(t+\eta) \frac{\partial \Omega}{\partial t}=-\frac{4t(t+\eta)(1-\eta(2t+\eta))}{\left( 1+(2t+\eta)^2 \right)^{3/2}\left( 1+\eta^2 \right)^{3/2}},\quad \left|(t+\eta)\frac{\partial \Omega}{\partial t}\right|\leq \frac{C}{1+\eta^2}\\
\label{estimrad6}
(t+\eta) \frac{\partial \Omega}{\partial r}=-\frac{4(t+\eta)^2(1+\eta(2t+\eta))}{\left( 1+(2t+\eta)^2 \right)^{3/2}\left( 1+\eta^2 \right)^{3/2}},\quad \left|(t+\eta)\frac{\partial \Omega}{\partial r}\right|\leq \frac{C}{1+\eta^2}.
\end{gather}
Combining \eqref{exp_dtu_dru} with the equalities in \eqref{estimrad1}, \eqref{estimrad2}, \eqref{estimrad3}, \eqref{estimrad4}, \eqref{estimrad5} and \eqref{estimrad6}, we see that 
\begin{equation}
 \label{lim_rad}
 \lim_{t\to\infty} (t+\eta)\frac{\partial u}{\partial t}(t,t+\eta)=-\lim_{t\to\infty} (t+\eta)\frac{\partial u}{\partial r}(t,t+\eta)=g'(\eta).
\end{equation}  
Using again \eqref{exp_dtu_dru}, together with the uniform bounds in \eqref{estimrad1}, \eqref{estimrad2}, \eqref{estimrad3}, \eqref{estimrad4}, \eqref{estimrad5} and \eqref{estimrad6} we obtain, uniformly for $t\geq -A\geq 0$, $\eta\geq A$,
\begin{equation}
 \label{boundCVD}
 |t+\eta| \left( \left|\frac{\partial u}{\partial t}\right|+ \left|\frac{\partial u}{\partial r}\right|\right)\leq \frac{C(A)}{1+\eta^2},
\end{equation} 
 where the constant $C(A)$ depends only on $A$, and we have used that $U$ and $\dot{U}$ are bounded on $\left[0,\pi/2-\arctan A\right]$. By \eqref{lim_rad}, \eqref{boundCVD} and dominated convergence, we obtain the conclusion of the lemma.
 \end{proof}
 \begin{proof}[Proof of Theorem \ref{thm:self0'}]
 In view of Lemma \ref{lem:radiation}, to prove  Theorem \ref{thm:self0'}, it is sufficient to find a solution $v_L$ of the linear wave equation \eqref{FW} such that 
\begin{multline*}
  \lim_{t\to\infty} \int_{A+t}^{\infty} \left(r\partial_t v_L(t,r)-g'(r-t)\right)^2dr\\
  =\lim_{t\to\infty} \int_{A+t}^{\infty} \left(r\partial_r v_L(t,r)+g'(r-t)\right)^2dr=0.
 \end{multline*}
 By explicit computation, one can check that the solution $v_L$ of \eqref{FW} with radial initial data $(v_0,v_1)$ is given by
 \begin{equation}
     \label{formulavL}
v_L(t,r)=\frac{1}{r}\left(F(t+r)-F(t-r)\right),\quad F(\eta)=\frac{\eta}{2}v_0(|\eta|)+\frac{1}{2}\int_0^{|\eta|} rv_1(r)dr.
 \end{equation}
As a consequence, if $(v_0,v_1)\in \mathcal{H}^1$, we have that $F'\in L^2(\mathbb{R})$ and
\begin{multline}
\label{lim_radiation'}
  \lim_{t\to\infty} \int_{A+t}^{\infty} \left(r\partial_t v_L(t,r)+F'(t-r)\right)^2dr\\
  =\lim_{t\to\infty} \int_{A+t}^{\infty} \left(r\partial_r v_L(t,r)-F'(t-r)\right)^2dr=0.
 \end{multline}
 In view of \eqref{lim_radiation}, \eqref{lim_radiation'}, we see that \eqref{asymptotic-linear} will hold if and only if $F'(\eta)=-g'(-\eta)$. This amounts to solving the system of equations, for $r>0$,
 \begin{equation*}
 \left\{
 \begin{aligned}
  \frac 12 v_0(r)+\frac 12 rv_0'(r)+\frac{rv_1(r)}{2}&=-g'(-r)\\
 \frac 12 v_0(r)+\frac 12 rv_0'(r)-\frac{rv_1(r)}{2}&=-g'(r),
\end{aligned}\right.
 \end{equation*}
which has the solution
\begin{equation}
\label{defv0v1}
 v_0(r)=\frac{1}{r}\left(g(-r)-g(r)\right),\quad v_1(r)=\frac{1}{r}\left( g'(r)-g'(-r) \right).
\end{equation} 
 As a conclusion, defining $v_L$ as the solution of \eqref{FW} with initial data $(v_0,v_1)$, we have that the conclusion \eqref{asymptotic-linear} of Theorem \ref{thm:self0'} is satisfied. Also, in view of the expansion \eqref{eq:T73} of $\arctan$, one has, as $\eta\to+\infty$,
 $$ U\left(\frac{\pi}{2}-\arctan \eta\right)=U(0)-\frac{1}{\eta}\dot{U}(0)+O\left( \frac{1}{\eta^2} \right),\quad \dot{U}\left(\frac{\pi}{2}-\arctan \eta\right)=\dot{U}(0)+O\left(\frac{1}{\eta}\right),$$
 so that
 $$ g(\eta)=-\frac{1}{\eta}U(0)+O\left(\frac{1}{\eta^2}\right),\quad g'(\eta)=\frac{1}{\eta^2}U(0)+O(\frac{1}{\eta^3}),\quad \eta\to\infty.$$
 Using again \eqref{eq:T73} and the estimates of $U$ and $\dot{U}$ at the blow-up time $T_+=\pi$ (see Lemma \ref{lem:asympt_duffing}), we obtain, as $\eta\to-\infty$,
 $$U\left(\frac{\pi}{2}-\arctan \eta\right)=\sqrt{2}\eta+O\left( \frac{1}{\eta} \right),\quad \dot{U}\left(\frac{\pi}{2}-\arctan \eta\right)=-\sqrt{2}\eta^2+O(1),$$
so that
$$ g(\eta)=-\sqrt{2}+O(1/\eta),\quad g'(\eta)=O\left( \frac{1}{\eta^2} \right),\quad\eta\to-\infty.$$
As a consequence, we see that, as $r\to\infty$,
$$ v_0(r)=-\frac{\sqrt{2}}{r}+O\left( \frac{1}{r^2} \right),\quad \partial_rv_0(r)=\frac{\sqrt{2}}{r^2}+O\left(\frac{1}{r^2}\right),\quad v_1(r)=O\left(\frac{1}{r^3}\right).$$
In particular, $v_0\in \dot{H}^s$, $s\in (1/2,1]$ but $v_0\notin L^3$ and thus $v_0\notin \dot{H}^{1/2}$. Also $v_1\in L^p$ for $p>1$ which implies that $v_1\in \dot{H}^{\nu}$ for $\nu\in (-3/2,0]$. Differentiating the formulas for $v_0$ and $v_1$ and using similar argument as in Subsection \ref{sub:rough} one can check also that $(v_0,v_1)\in \mathcal{H}^s$ for $s>1$, which concludes the proof of Theorem \ref{thm:self0'}. 
\end{proof}

\begin{remark}
\label{rem:transition}
Using \eqref{estimrad4} and the definition of $u$, we see that 
$$ \lim_{t\to\infty}(t+\eta)u(t,t+\eta)=-g(\eta).$$
On the other hand, by \eqref{formulavL} with $F(\eta)=g(-\eta)$, and using that $g$ converges to $-\sqrt{2}$ in $-\infty$, we obtain
$$ \lim_{t\to\infty}(t+\eta)v_L(t,t+\eta)=-\sqrt{2}-g(\eta).$$
As a conclusion,
$$ \lim_{t\to\infty}(t+\eta)\Big(u(t,t+\eta)-v_L(t,t+\eta)\Big)=\sqrt{2}.$$
This discrepancy between $u$ and the linear solution $v_L$ expresses the transition between the linear behaviour and $\sqrt{2}/t$ at the boundary of the wave cone.  
\end{remark}
\begin{remark}
\label{rem:nonradiative}
By Lemma \ref{lem:radiation}, a threshold solution $u$ satisfies, for $p\geq 1$:
\begin{multline*}
\forall A<B,\quad \lim_{t\to\infty}\int_{A+t<|x|<B+t}|\partial_tu(t,x)|^pdx\\
=\lim_{t\to\infty}\int_{A+t<|x|<B+t}|\partial_ru(t,x)|^pdx=4\pi\int_{A}^{B}(g'(\eta))^pd\eta.
\end{multline*}
We claim that this quantity is positive for all $A$, $B$ with $A<B$. In particular the solution is not nonradiative for $|x|>R+|t|$ for any $R>0$ (see \cite{CoDuKeMe22Pb} for the definition of nonradiative). Indeed, if $g'(\eta)=0$ on an interval $I$, solving a first order differential equation, we would obtain that $U(s)=\frac{c}{\sin s}$ for some constant $c\in \mathbb{R}$ on this interval.  Since $U$ is also a nonzero solution of $U''+U=U^3$, we must have $c\in \{\pm \sqrt{2}\}$, and (by uniqueness in the Cauchy-Lipschitz theorem), $U(s)=\frac{c}{\sin s}$ on all its domain of existence. Since $U$ is defined at $s=0$ this is an obvious contradiction. Note that the solutions $U(s)=\pm\sqrt{2}/\sin(s)$ of the Duffing equations correspond, undoing the Penrose transformation, to the solutions $u(t)=\pm\sqrt{2}/t$ of \eqref{NLW}; see the proof of Proposition~\ref{prop:blow_up_attractor} in subsection~\ref{sec:attractor}.

 As a consequence, we also obtain the existence of a constant $C>0$ such that 
\begin{equation}
    \label{eq:dtLp4}
    \int_{t-1}^{t+1}|\partial_tu(t,r)|^pr^2dr\geq \frac{1}{Ct^{p-2}},
\end{equation}
since in the range of the integral, $r\approx t$. This gives the lower bound in \eqref{eq:asymptoticdtLp}, concluding the proof of Theorem \ref{thm:self2}.
\end{remark}

\section{Blow-up}
In this section, we fix once and for all $(X, Y)\in \mathbb R^2$  such that $u=u_{X,Y}$ satisfies the case (Blow-up) of Theorem~\ref{thm:main}. Since $(X, Y)$ are fixed, we will write $u$ in place of $u_{X, Y}$. In particular, 
\begin{equation*}
    u=u(t, r) \text{ blows up at } r=0, t_+>0.    
\end{equation*}
Recall from Lemma~\ref{lem:penrose} that, with $r=\lvert x \rvert$, 
\begin{equation}\label{eq:penrose_recall_blowup_section}
    \begin{array}{ccc}
        u(t, x)=\Omega(t, r)U(s), &\text{where }&
            \left\{\begin{array}{c}
                \Omega(t, r)=2(1+(t+r)^2)^{-1/2}(1+(t-r)^2)^{-1/2}\\ 
                s=\arctan(t+r)+\arctan(t-r).
            \end{array}\right.
    \end{array}
\end{equation}
Considering $s\in (-\pi, \pi)$ as an independent variable, the function $U=U(s)$ satisfies $\ddot{U}+U=U^3$ on the interval $(T_-, T_+)$. Since we are in the (Blow-up) case, $T_+\in(0, \pi)$. For $t\ge 0$, the maximal existence domain of $u=u(t, r)$ is the region $t<M_+(T_+, r)$, where $t=M_+(T_+, r)$ is the positive root of the quadratic equation
\begin{equation}\label{eq:quadratic_blowup}         
    1+r^2-t^2=2t\cot(T_+),
\end{equation}
that is the surface $s(t, r)=T_+$; recall Figure~\ref{fig:hyperboloids}. 

In this section, we will first of all prove Theorem~\ref{thm:blowup_stable} (stability of blow-up). Then we will study the pointwise blow-up in the sense of Merle and Zaag, as announced in the introduction. We will then prove Theorem~\ref{thm:blowup} (asymptotics of Lebesgue and Sobolev norms) and finally we will prove Proposition~\ref{prop:blow_up_attractor} (convergence to an attractor).

\subsection{Proof of Theorem~\ref{thm:blowup_stable}}
  By Lemma~\ref{lem:asympt_duffing}, there is $T_0\in (0, T_+)$ such that $U(s), \dot U(s)$ and $\ddot U(s)$ have the same sign for $s\in (T_0, T_+)$. Without loss of generality, we assume 
\begin{equation}\label{eq:U_Udot_Uddot}
    \begin{array}{cccc}
    U(s)>0, & \dot{U}(s)>0, & \ddot{U}(s)>0, & \text{for all }s\in(T_0, T_+).
    \end{array}
\end{equation}

The main step of the proof of Theorem~\ref{thm:blowup_stable} is the construction of a suitable \emph{barrier function}. For a $\delta\in(0, \pi-T_+)$ to be chosen later, we let 
\begin{equation}\label{eq:u_delta}
    \begin{array}{cc}
        u_\delta(t, r)=\Omega(t, r)U(s-\delta), & s=\arctan(t+r)+\arctan(t-r).
    \end{array}
\end{equation}
Since $U_\delta(s)=U(s-\delta)$ clearly solves the Duffing ODE~\eqref{eq:DuffingIVP}, $u_\delta$ is a solution to the cubic wave equation~\eqref{NLW} that blows-up at $(t_+(\delta),0)$,  where $t_+(\delta)>t_+$ satisfies  
\begin{equation}\label{eq:t_plus_delta}
    2\arctan t_+(\delta) = T_++\delta.
\end{equation}
\begin{lemma}[Barrier function] \label{lem:barrier} 
Provided $\delta>0$ is small enough, there is $t_1\in (0, t_+)$, depending on $u$ and on $\delta$ and such that for all $r\in [0, t_+(\delta)-t_1]$, 
\begin{equation}\label{eq:barrier_conclusion}
    \begin{array}{cc}
        u(t_1, r)>u_\delta(t_1, r), &(\partial_t+\partial_r)(ru(t, r))|_{t=t_1}>(\partial_t+\partial_r)(ru_\delta(t, r))|_{t=t_1}.
    \end{array}
\end{equation}
\end{lemma}
\begin{proof}
    We will find a $t_1$ that satisfies the second condition in~\eqref{eq:barrier_conclusion}. Once this is done, checking that the first condition is also satisfied is immediate. We compute 
    \begin{equation}\label{eq:compute_udelta_deriv}
        \begin{split}
            &(\partial_t+\partial_r)(ru(t, r)-ru_\delta(t, r))= \\ 
            &\frac{\Omega(t, r)}{1+(t+r)^2}\left( (1+t^2-r^2)(U(s)-U(s-\delta))+2r(\dot{U}(s)-\dot{U}(s-\delta))\right).
        \end{split}
    \end{equation}
    Recall that $U=U(s)$ and $\dot{U}=\dot{U}(s)$ are increasing for $s\in (T_0, T_+)$, corresponding to the region depicted in Figure~\ref{fig:barrier_function}. 
    \begin{figure}
    \centering
    \begin{tikzpicture}
        \begin{axis}[axis on top, axis x line=middle, axis y line=middle, xmin=-0.61, xmax=4, ymin=-0.61, ymax=3.34, xlabel=$\scriptsize{r}$, ylabel=$\scriptsize{t}$, xtick=\empty, ytick=\empty]
            \addplot[mark=none, draw=blue, domain=0:7, name path=down] {-1.2+sqrt(1+x^2+(1)^2)};
            \draw (0, -1.2+1.44) node[anchor=east]{$t_0$};
            \draw[blue] (4, 1.8) node[anchor=east] {$s(t,r)=T_0$};
            \addplot[mark=none, draw=blue, domain=0:3, name path=down] {0.8+sqrt(1+x^2+(1)^2)};
            \draw[blue] (2.48, 3.3) node[anchor=north] {$s(t,r)=T_+$};
            \draw (0, 0.8+1.44) node[anchor=east]{$t_+$};
            \draw[dashed] (0, 3) node[anchor=east]{$t_+(\delta)$} -- (3.3, -0.3) node[anchor=east]{$r=t_+(\delta)-t$};
            
            \draw (0, 1.5) node[anchor=east]{$t_1$}--(1.5, 1.5);
        \end{axis}
        \end{tikzpicture}
    \caption{The time $t_1$ is chosen so that the segment lies well inside the region $s\in (T_0, T_+)$. The corresponding forward light cone has its tip at $(t_+(\delta), 0)$, where $u_\delta$ blows up.} 
    \label{fig:barrier_function}
\end{figure}
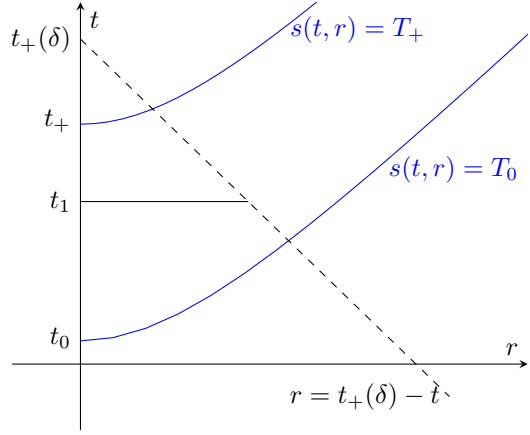
So it is sufficient to find a $t_1$, with $0<t_+(\delta)-t_1<1$, such that for $0<r<t_+(\delta)-t_1$ it is true that 
    \begin{equation}\label{eq:s_minus_delta}
        s(t_1, r)-\delta= \arctan(t_1+r)+\arctan(t_1-r)-\delta >T_0.
    \end{equation}
    
    Note that the condition $t_+(\delta)-t_1<1$ ensures that all coefficients in the right-hand side of~\eqref{eq:compute_udelta_deriv} are positive in this range of $r$ for $t=t_1$. 
    
    Since $\arctan(t_1+r)+\arctan(t_1-r)$ is decreasing in $r$, it is sufficient to check 
    \begin{equation}\label{eq:barrier_tocheck}
        \arctan(t_+(\delta))+\arctan(2t_1-t_+(\delta))-\delta> T_0.
    \end{equation}
    Obviously, $2\arctan(t_+(\delta))-\delta\to T_+>0$ as $\delta\to 0$. Thus there is a $\delta_0>0$ such that, for all $\delta\in (0, \delta_0)$,
    \begin{equation*}
         \arctan(t_+(\delta))+\arctan(2t_+-t_+(\delta))-\delta> T_0,
    \end{equation*}
    so taking $t_1$ close enough to $t_+$ we see that~\eqref{eq:barrier_tocheck} is satisfied, concluding the proof.
\end{proof}
We are now ready to complete the proof of Theorem~\ref{thm:blowup_stable}. Let $\eta>0$ be fixed. Let moreover $v=v(t, r)$ be a general radial solution to~\eqref{NLW}, not necessarily of the form~\eqref{eq:uAB_initial}, such that 
\begin{equation}\label{eq:u_close_v_blowup}
    \lVert \boldsymbol{u}(0)-\boldsymbol{v}(0)\lVert_{H^3\times H^2}<\epsilon.
\end{equation}
Recall that we need to prove that, if $\epsilon=\epsilon(u, \eta)>0$ is sufficiently small, then there is $t_+'<t_++\eta$ such that $\lVert \boldsymbol{v}(t)\rVert_{H^3\times H^2}\to \infty$ as $t\nearrow t_+'$.

We fix a $\delta=\delta(u, \eta)>0$ so small that $t_+(\delta)<t_++\eta$. We let $t_1$ and $u_\delta$ be as in Lemma~\ref{lem:barrier}; thus there is $\alpha=\alpha(u, \delta)>0$ such that 
\begin{equation}\label{eq:alpha_blowup}
    \begin{array}{cc}
        u(t_1, r)\ge u_\delta(t_1, r)+\alpha, &(\partial_t+\partial_r)(ru(t, r))|_{t=t_1}\ge (\partial_t+\partial_r)(ru_\delta(t, r))|_{t=t_1}+\alpha.
    \end{array}
\end{equation}
By standard local-wellposedness theory in $H^3\times H^2$, there is $\epsilon_0=\epsilon_0(u, \alpha)$ such that, for all $\epsilon <\epsilon_0$, the condition~\eqref{eq:u_close_v_blowup} implies that $\lVert \boldsymbol{u}(t_1) - \boldsymbol{v}(t_1)\rVert_{H^3\times H^2}\le \frac{\alpha}{C}$ for a large constant $C>0$. By Sobolev embedding,
\begin{equation}\label{eq:sobolev_embedding}
    \lvert u(t_1, r)-v(t_1, r)\rvert + \big\lvert (\partial_t+\partial_r)(ru(t, r)-rv(t, r))|_{t=t_1} \big\rvert \le \frac\alpha2  
\end{equation}
for $r\in [0, t_+(\delta)-t_1]$, and we conclude that 
\begin{equation}\label{eq:v_vs_barrier}
    \begin{array}{cc}
      v(t_1, r)\ge u_\delta(t_1, r), & (\partial_t+\partial_r)(rv(t, r))|_{t=t_1}\ge (\partial_t+\partial_r)(ru_\delta(t, r))|_{t=t_1}
    \end{array}
\end{equation}
again for $r\in [0, t_+(\delta)-t_1]$. Now, by positivity properties of the wave equation on $\mathbb R^{1+3}$, these last conditions imply 
\begin{equation}\label{eq:v_vs_barrier_conclusion}
    v(t, r)\ge u_\delta(t, r)
\end{equation}
on the truncated light cone  $\{ t\in (t_1, t_+(\delta)),\ r\in [0, t_+(\delta)-t]\}$; see Figure~\ref{fig:barrier_function}. 
Since $\lVert u_\delta(t, \cdot)\rVert_{L^\infty(\mathbb R^3)}\to \infty$ as $t\nearrow t_+(\delta)$, we have, again by Sobolev embedding,
\begin{equation}\label{eq:v_blows_up}
    \lVert \boldsymbol{v}(t)\rVert_{H^3\times H^2}\ge C\lVert v(t, \cdot)\rVert_{L^\infty(\mathbb R^3)}\to \infty, \text{ as }t\nearrow t_+(\delta),
\end{equation}
concluding the proof of Theorem~\ref{thm:blowup_stable}.

\subsection{Blow-up profiles}\label{sec:blowup_profile}
We consider $t_\ast>0$ and $r_\ast\ge 0$ such that the point $(t_\ast, r_\ast, 0,0)\in\mathbb R^{1+3}$ lies on the blow-up surface $s(t_\ast, r_\ast)=T_+$, that is $t_\ast=M_+(T_+, r_\ast)$; recall~\eqref{eq:quadratic_blowup}. We pointed out in Remark~\ref{rem:characteristic_points} that this point is non-characteristic. We introduce the \emph{radial self-similar coordinates} on the backwards light cone at $(t_\ast, r_\ast)$:
\begin{equation}\label{eq:SSCoords}
    \begin{array}{ccc}
        \displaystyle y=\frac{r-r_\ast}{t_\ast -t}\in (-1, 1), & \sigma=-\log(t_\ast -t )\ge 0, & w(\sigma, y)=(t_\ast -t)u(t, r).
    \end{array}
\end{equation}
%
In this subsection we will prove that 
\begin{equation}\label{eq:blowup_profile}
    \begin{array}{cc}\displaystyle
        \lim_{\sigma\to \infty} w(\sigma, y)=\sqrt{2}\frac{(1-d^2)^\frac12}{1+yd}, & \text{where } \displaystyle d=\left.\frac{\partial M_+}{\partial r}(T_+, r)\right|_{r=r_\ast};
    \end{array}
\end{equation}
that is, our blow-up solutions follow exactly the behaviour predicted by Merle--Zaag, generalizing their one-dimensional case; see~\cite[p.3]{MeZa15}. 
\begin{remark}
    In the general, non-radial framework of Merle--Zaag, our $y$ coordinate should be replaced by $\vec y=\frac{x-(r_\ast,0,0)}{t_\ast -t}$, and the relation~\eqref{eq:blowup_profile} should read
    \begin{equation}\label{eq:blowup_profile_general}
        \begin{array}{cc}\displaystyle
        \lim_{\sigma\to \infty} w(\sigma, \vec y)=\sqrt{2}\frac{(1-\lvert \vec{d}\rvert^2)^\frac12}{1+\vec d\cdot \vec y}, & \text{where } \displaystyle \vec d=\left.\nabla_x(M_+(T_+, \lvert x \rvert))\right|_{x=(r_\ast,0,0)}.
    \end{array}
    \end{equation}
    However, by the radial symmetry of our problem, proving~\eqref{eq:blowup_profile} is clearly enough to conclude that~\eqref{eq:blowup_profile_general} holds. 
\end{remark}
To prove~\eqref{eq:blowup_profile}, recalling $u(t, r)=\Omega(t, r)U(s)$, we compute the following asymptotic for $s=\arctan(t+r)+\arctan(t-r)$, uniform in $|y|<1$:
\begin{equation}\label{eq:s_expansion}
    s=T_+-e^{-\sigma}\left( c(t_\ast, r_\ast)+d(t_\ast, r_\ast)r_\ast y\right) + O(e^{-2\sigma}), 
\end{equation}
where
\begin{equation}\label{eq:c_and_d_blowup}
    \begin{array}{cc}\displaystyle
        c=\frac{1}{1+(t_\ast+r_\ast)^2}+\frac{1}{1+(t_\ast-r_\ast)^2}, & \displaystyle d=\frac1r_\ast\left( \frac{1}{1+(t_\ast+r_\ast)^2}-\frac{1}{1+(t_\ast-r_\ast)^2}\right)\!.
    \end{array}
\end{equation}
So the ODE asymptotics of Lemma~\ref{lem:asympt_duffing} yield 
\begin{equation}\label{eq:U_asympt_blowup}
    U(s)=\sqrt 2\big(c(t_\ast, r_\ast)+d(t_\ast, r_\ast)rr_\ast\big)^{-1}e^\sigma + O(1).
\end{equation}
Recalling that $\Omega(t, r)=2(1+(t+r)^2)^{-1/2}(1+(t-r)^2)^{-1/2}$, we have 
\begin{equation}\label{eq:blowup_profile_making}
    \begin{split}
        w(\sigma, y)&=\sqrt{2}\frac{\Omega(t_\ast, r_\ast)}{c(t_\ast, r_\ast)+d(t_\ast, r_\ast)r_\ast y} +O(e^{-\sigma}) \\ 
        &=\sqrt{2}\frac{ (1+(t_\ast + r_\ast)^2)^{1/2}(1+(t_\ast -r_\ast)^2)^{1/2}}{1+t_\ast^2-r_\ast^2}\frac{1}{1+yd} + O(e^{-\sigma})\\ 
        &=\sqrt{2}\frac{(1-d^2)^{1/2}}{1+yd} + O(e^{-\sigma}), 
    \end{split}
\end{equation}
%
with 
\begin{equation}\label{eq:d_vec}
    d=\frac{d(t_\ast, r_\ast)}{c(t_\ast, r_\ast)}r_\ast=\frac{2t_\ast r_\ast}{1+t_\ast^2 + r_\ast^2}.
\end{equation}
Differentiating the relation~\eqref{eq:quadratic_blowup} that defines $M_+(T_+, r)$ we immediately infer that
\begin{equation}\label{eq:d_is_derivative_final}
    \left.\frac{\partial M_+(T_+, r)}{\partial r}\right|_{r=r_\ast}=\frac{r_\ast}{t+\cot(T_+)}=\frac{2t_\ast r_\ast}{1+t_\ast^2+r_\ast^2}= d, 
\end{equation}
where we used~\eqref{eq:quadratic_blowup} to obtain the second identity. This completes the proof of~\eqref{eq:blowup_profile}.

\subsection{Proof of Theorem~\ref{thm:blowup}} We start by proving~\eqref{eq:blowup_Lthree}, that is
\begin{equation}\tag{\ref{eq:blowup_Lthree}}\label{eq:blowup_Lthree_recall}
    \lVert u(t)\rVert_{L^3(\mathbb R^3)}= \frac{C_0^\frac13}{(t_+-t)^{1/2}}+O(1).
\end{equation}
\noeqref{eq:blowup_Lthree_recall}
Recalling $s(t, r)=\arctan(t+r)+\arctan(t-r)$, since $s(t_+,0)=T_+$, a Taylor expansion of $\arctan$ around $t_+$ yields
\begin{equation}\label{eq:TaylorExpandBlowup}
    T_+- s=\frac{2}{1+t_+^2}(t_+-t)+\frac{2t_+}{(1+t_+^2)^2}r^2+O((t_+-t)^2+r^4).
\end{equation}
By the asymptotic $U(s)=\sqrt{2}(T_+-s)^{-1}+O(T_+-s)$ of Lemma~\ref{lem:asympt_duffing},
\begin{equation}\label{eq:integrand_asymptotic}
    \lvert U(s)\rvert^3=\frac{(1+t_+^2)^3}{2\sqrt 2} \left( t_+-t+\frac{t_+}{1+t_+^2}r^2\right)^{-3} + O( (t_+-t+r^2)^{-2}).
\end{equation}
Now we expand $\Omega(t, r)=2(1+(t-r)^2)^{-1/2}(1+(t+r)^2)^{-1/2}$ around $(t_+, 0)$, giving
\begin{equation}\label{eq:OmegaExpansionBlowup}
    \Omega^3(t, r)=\frac{2^3}{(1+t_+^2)^3}+O(t_+-t+r), 
\end{equation}
and we finally conclude the main asymptotic
\begin{equation}\label{eq:main_asymptotic_blowup}
    \lvert \Omega U\rvert^3=2\sqrt 2\left( t_+-t+ \frac{t_+ r^2}{1+t_+^2}\right)^{-3} +O((t_+-t+r^2)^{-2}).
\end{equation}
Since $\lVert u(t)\rVert_{L^3(\mathbb R^3)}^3=4\pi\int_0^\infty \lvert \Omega U\rvert^3r^2\, dr, $ we split the integral as follows. 
\begin{multline*}
    \int_0^\infty \lvert \Omega U\rvert^3r^2\, dr=2\sqrt{2}\int_0^{\varepsilon}\left( t_+-t+ \frac{t_+ r^2}{1+t_+^2}\right)^{-3}r^2dr\\
    +O\left(\int_0^{\varepsilon}\left( t_+-t+ \frac{t_+ r^2}{1+t_+^2}\right)^{-2}r^2dr\right)+\int_{\varepsilon}^\infty \lvert \Omega U\rvert^3r^2\, dr 
\end{multline*}
 Using $\Omega^3\lesssim r^{-6}$ and $\lvert U(s)\rvert^2\lesssim 1$, for $r\geq \varepsilon$, we deduce 
\begin{multline*}
    \int_0^\infty \lvert \Omega U\rvert^3r^2\, dr=2\sqrt{2}\int_0^{+\infty}\left( t_+-t+ \frac{t_+ r^2}{1+t_+^2}\right)^{-3}r^2dr+O(1)\\
    +O\left(\int_0^{\infty}\left( t_+-t+ \frac{t_+ r^2}{1+t_+^2}\right)^{-2}r^2dr\right). 
\end{multline*}
By the change of variable $r=\sqrt{t_+-t}\rho$, we have, for $p>3/2$
\begin{equation}
\label{CoV}
\int_0^{+\infty}\left( t_+-t+ \frac{t_+ r^2}{1+t_+^2}\right)^{-p}r^2dr=\left(t_+-t\right)^{\frac 32-p} \int_0^{\infty} \left(1+\frac{t_+\rho^2}{1+t_+^2}\right)^{-p}\rho^2d\rho.
\end{equation}
This proves the bound~\eqref{eq:blowup_Lthree_recall} with constant 
\begin{equation}\label{eq:constant_blowup}
    C_0=8\sqrt 2 \pi \int_0^\infty \left(1+\frac{t_+\rho^2}{1+t_+^2}\right)^{-3} \rho^2\, d\rho,
\end{equation}
as we wanted. To prove~\eqref{eq:blowup_sobolev}, which we recall:
\begin{equation}\tag{\ref{eq:blowup_sobolev}}\label{eq:blowup_sobolev_recall}
    \lVert u(t)\rVert_{\dot{H}^{1/2}} \approx \frac{1}{(t_+-t)^{1/2}};
\end{equation}
we start by pointing out the Sobolev embedding $\lVert u(t)\rVert_{\dot{H}^{1/2}}\gtrsim \lVert u(t)\rVert_{L^3(\mathbb R^3)}$, which by~\eqref{eq:blowup_Lthree_recall} gives one of the inequalities in~\eqref{eq:blowup_sobolev_recall}. To prove the opposite inequality, we consider the interpolation $\lVert u(t)\rVert_{\dot{H}^{1/2}}^2\le \lVert u(t)\rVert_{L^2}\lVert \nabla u(t)\rVert_{L^2}$. The asymptotics~\eqref{eq:integrand_asymptotic} and \eqref{eq:OmegaExpansionBlowup}, together with \eqref{CoV}, yield
\begin{equation}\label{eq:blowup_Ltwo}
    \begin{split}
        \lVert u(t)\rVert_{L^2(\mathbb R^3)}^2=4\pi\int_0^{\infty}\lvert \Omega U\rvert^2r^2\, dr= O((t_+-t)^{-\frac12}).
    \end{split}
\end{equation}
On the other hand, 
\begin{equation}\label{eq:blowup_Hone}
    \lVert \nabla u(t)\rVert_{L^2(\mathbb R^3)}^2\lesssim \int_0^\infty \left\lvert\frac{\partial \Omega}{\partial r} U\right\rvert^2r^2\, dr + \int_0^\infty \left\lvert\Omega \dot U\frac{\partial s}{\partial r}\right\rvert^2r^2\, dr;
\end{equation}
and the first of the latter two integrals is bounded by $O((t_+-t)^{-1/2})$ just like~\eqref{eq:blowup_Ltwo}. To estimate the second integral,  Lemma~\ref{lem:asympt_duffing} gives 
\begin{equation}\label{eq:recall_asympt_duffing}
    \dot{U}(s)=\sqrt{2}(T_+-s)^{-2}+O(1),
\end{equation}
and since $\left\lvert\frac{\partial s}{\partial r}\right\rvert=4\lvert t\rvert r/[(1+(t+r)^2)(1+(t-r)^2)]\le 4t_+ r$, 
\begin{equation}\label{eq:conclude_sobolev_blowup}
    \begin{split}
        \int_0^\infty \left\lvert\Omega \dot U(s)\frac{\partial s}{\partial r}\right\rvert^2r^2\, dr &\le 4t_+\int_0^\infty \left\lvert\Omega \dot U(s)\right\rvert^2r^4\, dr \\
        &\le O((t_+-t)^{-\frac32}),
    \end{split}
\end{equation}
by similar computations as above.
So by~\eqref{eq:blowup_Hone} $\lVert \nabla u (t)\rVert_{L^2}^2\le O((t_+-t)^{-3/2})$. We conclude 
\begin{equation}\label{eq:blowup_conclusion}
    \lVert u(t)\rVert_{\dot{H}^{1/2}}^2\le \lVert u(t)\rVert_{L^2}\lVert \nabla u(t)\rVert_{L^2} \le O((t_+-t)^{-1}), 
\end{equation}
which is the desired opposite inequality. This concludes the proof of~\eqref{eq:blowup_sobolev_recall}.

We turn now to~\eqref{eq:blowup_ut_Lthreetwo}, which we recall: 
\begin{equation}\tag{\ref{eq:blowup_ut_Lthreetwo}}\label{eq:blowup_ut_Lthreetwo_recall}
\lVert \partial_t u(t)\rVert_{L^\frac{3}{2}(\mathbb R^3)}= \frac{2^{-\frac12}C_0^\frac23}{t_+-t}+O(1).
\end{equation}
We have that 
\begin{equation}\label{eq:ut_splitting}
    \lVert \partial_t u(t)\rVert_{L^\frac{3}{2}(\mathbb R^3)}^{\frac32}= 4\pi\int_0^\infty \left\lvert\frac{\partial \Omega}{\partial t} U+\Omega \dot U\frac{\partial s}{\partial t}\right\rvert^\frac32r^2\, dr.
\end{equation}
The  previous method shows that $\int_0^\infty \left\lvert\frac{\partial \Omega}{\partial t} U(s)\right\rvert^\frac32 r^2\, dr= O(1)$. So this term will be negligible compared to the other one. Now, by ~\eqref{eq:recall_asympt_duffing}, 
\begin{equation}\label{eq:Udot_isU}
    \lvert U(s)\rvert^3=2^\frac34\lvert \dot U(s)\rvert^\frac32 + O((T_+-s)^{-1}), 
\end{equation}
and since $\frac{\partial s}{\partial t}=2(1+t^2-r^2)/((1+(t+r)^2)(1+(t-r)^2))$, we have 
\begin{equation}\label{eq:partial_s_is_Omega}
    \left\lvert \Omega \frac{\partial  s}{\partial t}\right\rvert^\frac32 = \frac{2^3}{(1+t_+^2)^3}+O(t_+-t + r)=\lvert\Omega\rvert^3+O(t_+-t+r), 
\end{equation}
so we conclude that our integrand essentially coincides with the previous one to main order:
\begin{equation}\label{eq:u_t_final_asymptotic}
    \left\lvert \Omega \frac{\partial s}{\partial t} \dot{U}\right\rvert^\frac32 =2^{-\frac34}\lvert \Omega U\rvert^3 +O((t_+-t+r^2)^{-2}).
\end{equation}
In light of~\eqref{eq:main_asymptotic_blowup},the exact same computations as in the proof of~\eqref{eq:blowup_Lthree_recall} now yield 
\begin{equation}\label{eq:final_blowup_ut}
    4\pi\int_0^\infty \left\lvert \Omega \frac{\partial s}{\partial t} \dot{U}\right\rvert^\frac32 r^2\, dr = \frac{2^{-\frac34}C_0}{(t_+-t)^\frac32} +O((t_+-t)^{-\frac{1}{2}}), 
\end{equation}
and~\eqref{eq:blowup_ut_Lthreetwo_recall} follows from~\eqref{eq:ut_splitting}.

Finally,~\eqref{eq:blowup_ut_sobolev} is an immediate consequence of the embedding $L^{3/2}\subset \dot{H}^{-1/2}$. 
\subsection{Proof of Proposition~\ref{prop:blow_up_attractor}}\label{sec:attractor} The ODE asymptotics of Lemma~\ref{lem:asympt_duffing} immediately imply
\begin{equation}\label{eq:asymptotics_with_sine}
    U(s)=\frac{\sqrt 2}{\sin(T_+-s)}+O(T_+-s).
\end{equation}
Using the formulas~\eqref{eq:inverse_Penrose}, this yields for $u(t, r)=\Omega(t, r)U(s)$ the asymptotic 
\begin{equation}\label{eq:asymptotics_with_sine_physical}
    u(t, r)=\frac{2\sqrt{2}}{a(1+r^2-t^2)-2bt} + \Omega(t, r)\cdot O(T_+-s),
\end{equation}
with $a=\sin T_+, b=\cos T_+$. The estimate $\Omega(t, r)\cdot O(T_+-s)=O(t_+-t)$ follows at once from the boundedness of $\Omega$ and from the Taylor expansion~\eqref{eq:TaylorExpandBlowup}. The proof is complete.
\printbibliography
 \end{document}